\documentclass[a4paper,12pt]{article}
\usepackage{graphicx,color,amssymb,amsthm,amsmath}

\newtheorem{theorem}{Theorem}[section]
\newtheorem*{theorem*}{Theorem}

\newtheorem{lemma}[theorem]{Lemma}
\newtheorem{proposition}[theorem]{Proposition}
\newtheorem*{proposition*}{Proposition}
\newtheorem{remark}[theorem]{Remark}

\begin{document}

\title{Titchmarsh--Weyl formula for the \\spectral density of a class of \\Jacobi matrices in the critical case}
\date{}
\author
    {
        Serguei Naboko\thanks
        {
            Department of Mathematical Physics, Institute of Physics, Saint-Petersburg State University, Ulianovskaia 1, St. Petergoff, Saint-Petersburg, 198904 Russia
            sergey.naboko@gmail.com
        }
        \and
        Sergey Simonov\thanks
        {
            St. Petersburg Department of V. A. Steklov Institute of Mathematics of the Russian Academy of Sciences, Fontanka 27, St. Petersburg, 191023 Russia;
            St. Petersburg State University, Universitetskaya nab. 7--9, St. Petersburg, 199034 Russia.
            sergey.a.simonov@gmail.com
        }
    }
\maketitle
\numberwithin{equation}{section}

    \begin{abstract}
    We consider a class of Jacobi matrices with unbounded entries in the so called critical (double root, Jordan box) case. We prove a formula for the spectral density of the matrix which relates its spectral density to the asymptotics of orthogonal polynomials associated with the matrix.
    \end{abstract}



\section{Introduction}

In the paper \cite{Aptekarev-Geronimo-2016} A. I. Aptekarev and J. S. Jeronimo for a class of Jacobi matrices with unbounded entries found a formula for the spectral density in terms of asymptotics of the orthogonal polynomials associated with the matrix. Their class is defined by certain summability conditions for the coefficients and is an example of the situation which can be called the {\it non-critical case}. The {\it critical case} for Jacobi matrices with unbounded entries was studied in several papers, among them \cite{Damanik-Naboko-2007,Janas-2006,Janas-Naboko-Sheronova-2009,Naboko-Simonov-2010}. The distinction between the cases is determined by the limit (whenever it exists) of the transfer-matrix of the eigenvector equation: in the non-critical case it is di\-a\-go\-na\-li\-zable, while in the critical case it is similar to a Jordan box. Note that for the special case of discrete Schr\"odinger operator with fast decreasing potential (the main diagonal) we have the critical situation only for two values of the spectral parameter $\lambda=\pm2$. However, for Jacobi matrices in the critical case we have the similar situation for all values of the spectral parameter $\lambda$. The types of asymptotics of generalized eigenvectors differ significantly in these two cases: critical and non-critical. In the papers dealing with the critical case mentioned above only asymptotics of the generalized eigenvectors were studied, and one cannot find an analog of the formula for the spectral density from \cite{Aptekarev-Geronimo-2016} for the critical case. The aim of the present paper is to fill this gap.

We consider a class of Jacobi matrices in the critical case which is an extension of the class studied in \cite{Janas-Naboko-Sheronova-2009}. Besides finding the asymptotics of generalized eigenvectors (in a more general case) we obtain a formula for the spectral density of the matrix in terms of the asymptotics of its orthogonal polynomials. Namely, we consider matrices
\begin{equation}\label{Jacobi matrix}
    \mathcal J=
    \left(%
    \begin{array}{ccccc}
    b_1 & a_1 & 0 & 0 & \cdots \\
    a_1 & b_2 & a_2 & 0 &\cdots \\
    0 & a_2 & b_3 & a_3 &\cdots \\
    0 & 0 & a_2 & b_4 & \cdots \\
    \vdots & \vdots & \vdots & \vdots & \ddots \\
    \end{array}%
    \right)
\end{equation}
with entries
\begin{equation}\label{entries}
    a_n=n^{\alpha}+p_n,\quad b_n=-2n^{\alpha}+q_n,
\end{equation}
where $a_n>0$, $b_n\in\mathbb R$,
\begin{equation*}
    \alpha\in(0,1)
\end{equation*}
and
\begin{equation}\label{p-q}
    \left\{\frac{p_n}{n^{\frac{\alpha}2}}\right\}_{n=1}^{\infty},\left\{\frac{q_n}{n^{\frac{\alpha}2}}\right\}_{n=1}^{\infty}\in l^1.
\end{equation}
The operators of this family act in the Hilbert space $l^2(\mathbb N)$ and have the following spectral properties: the spectrum on the right half-line is pure point, the spectrum on the left half-line is purely absolutely continuous. This follows from the asymptotics of the generalized eigenvectors by the subordinacy theory \cite{Gilbert-Pearson-1987,Khan-Pearson-1992}. For $\alpha\in(\frac12,\frac23)$  such a family was considered in \cite{Janas-Naboko-Sheronova-2009} and for $\lambda\in\mathbb R\backslash\{0\}$ asymptotics of generalized eigenvectors were found. In the present paper to find asymptotics we use a modified version of the remarkable method of R.-J. Kooman \cite{Kooman-2007} which can yield the result for all complex $\lambda$ in the same manner and works for $\alpha\in(0,1)$. In essence, the Kooman's method is based on a transformation which reduces the discrete linear system to the Levinson (L-diagonal) form and can be considered as an extension of Benzaid--Harris--Lutz methods \cite{Benzaid-Lutz-1987,Harris-Lutz-1975}, see also \cite{Bodine-Lutz-2015,Janas-Moszynski-2003,Silva-2007}. All of them are based on the discrete analog of the Levinson asymptotic theorem \cite[Theorem 8.1]{Coddington-Levinson-1955}.

We call a formula, which relates the spectral density of an ordinary dif\-fe\-ren\-tial or dif\-fe\-rence operator to the coefficient in asymptotics of the solution of the eigenvector equation which satisfies the boundary condition, the {\it Titchmarsh--Weyl formula} by analogy with the classical case of the Schr\"odinger operator on the half-line with a self-adjoint boundary condition at zero and a summable potential, see the book of E. C. Titchmarsh \cite[Ch. V, 5.6]{Titchmarsh-1962}.

In this study of formulas for the spectral density we have one application in mind, namely, the phenomenon of spectral phase transition. If a family of self-adjoint operators depends on one or several real parameters, it can happen that the space of these parameters is divided into regions where the operators have similar spectral structures. For example, in some regions the spectrum can be purely absolutely con\-ti\-nu\-ous, in others it can be discrete. Then on the boundaries of the regions happens a {\it spectral phase transition}. We want to see by explicit examples how such transitions happen in terms of the spectral measure. Some examples of spectral phase transitions can be found in the papers \cite{Janas-Naboko-2001, Janas-Naboko-2002, Simonov-2007}. In these works only the ``geometry'' and type of the spectrum were considered for the lack of suitable methods to analyze the spectral measure. This is where formulas for the spectral density could be useful. Several papers were devoted to establishing such formulas in both discrete and continuous cases, \cite{Janas-Simonov-2010,Kurasov-Simonov-2014,Simonov-2009}. That analysis has been used to study the behavior of the spectral density of discrete \cite{Simonov-2012} and differential \cite{Naboko-Simonov-2012,Simonov-2016} Schr\"odinger operators with the Wigner--von Neumann potential near the critical points which appear due to that form of the potential. These formulas were derived for special classes of operators, and, moreover, for the non-critical case. The example that we consider in the present paper is taken from a family of Jacobi matrices demonstrating a spectral phase transition, in the case when the parameters belong the boundary of two regions (a line in that case). On that boundary ($d=\pm1$, see \eqref{app b limits of entries}) the Jacobi matrices are in the critical case. The formula for the spectral density in this case is needed as the first step to understanding the ``inner structure'' of the spectral phase transition in this family.

We should mention interesting recent works by G. \'Swiderski \cite{Swiderski-2016,Swiderski-Trojan-2017,Swiderski-2017,Swiderski-2018} also devoted, in particular, to the study of the spectral density. However, these considerations did not include the Jordan box case.

The paper is organized as follows. In Section \ref{section prelim} we recall some basic notions and facts related to Jacobi matrices and their generalized eigenvector equations and also explain why in our situation the critical case occurs. In Section \ref{section stabilized} we describe the idea of approximating the matrix $\mathcal J$ by the ``stabilized'' matrices $\mathcal J_N$, also used in \cite{Aptekarev-Geronimo-2016}, which allows to find the spectral density of $\mathcal J$ exploiting the $*$-weak convergence of spectral measures. It contains two elementary propositions to be used in the proof of the main theorem. In Section \ref{section main} the main result of the paper (Theorem \ref{main theorem}) is proved: the formula for the spectral density of $\mathcal J$ in the critical case. The paper contains two appendices. In Appendix A we give a proof of the formula for the spectral density of the ``stabilized'' matrix adjusted to our form of it. And finally, in Appendix B we revisit the result of Aptekarev--Geronimo \cite{Aptekarev-Geronimo-2016} that gives a formula for the spectral density for a class of unbounded Jacobi matrices in the non-critical case. In the second appendix the proof uses the similar technique which was elaborated for the more complicated critical case.

\section{Preliminaries}\label{section prelim}

For the  complete definitions of a Jacobi matrix, orthogonal polynomials associated to it and its Weyl function in the limit-point case we refer to the book of N. I. Akhiezer \cite{Akhiezer-1965}.

The operator $\mathcal J$, defined by the matrix \eqref{Jacobi matrix}--\eqref{p-q}, according to the Carleman condition \cite{Akhiezer-1965} is self-adjoint in $l^2(\mathbb N)$. If $\mathcal E$ is its projection-valued spectral measure, $\{e_n,\ n\in\mathbb N\}$ is the standard basis in $l^2(\mathbb N)$, then $\rho=(\mathcal Ee_1,e_1)$ is its scalar {\it spectral measure} and $\rho'$, which exists and is finite a.e., is its {\it spectral density}. By $m$ we denote the {\it Weyl function} for which the following relations hold:
\begin{equation}\label{}
    m(\lambda)=\int_{\mathbb R}\frac{d\rho(x)}{x-\lambda},\quad\lambda\in\mathbb C\backslash\mathbb R,
\end{equation}
\begin{equation}\label{}
    \rho'(\lambda)=\frac1{\pi}{\rm Im\,}m(\lambda+i0),\quad\text{a. a. }\lambda\in\mathbb R.
\end{equation}
Orthogonal polynomials of the first and the second kind, $P_n(\lambda)$ and $Q_n(\lambda)$, respectively, are solutions of the eigenvector equation
\begin{equation}\label{eigenvector equation}
    a_{n-1}u_{n-1}+b_nu_n+a_nu_{n+1}=\lambda u_n,\quad n\geqslant 2,
\end{equation}
and have the initial values $P_1(\lambda)=1$, $P_2(\lambda)=\frac{b_1-\lambda}{a_1}$, $Q_1(\lambda)=0$, $Q_2(\lambda)=\frac{1}{a_1}$. For $\lambda\in\mathbb C\backslash\mathbb R$ their linear combination $Q_n(\lambda)+m(\lambda)P_n(\lambda)$ is the only (up to multiplication by a constant) solution of \eqref{eigenvector equation} which belongs to $l^2(\mathbb N)$. The weighted {\it Wronskian} of two solutions $u$ and $v$ of the eigenvector equation \eqref{eigenvector equation} is defined as
\begin{equation}\label{Wronskian}
    W\{u,v\}:=a_n(u_nv_{n+1}-u_{n+1}v_n),\quad n\in\mathbb N,
\end{equation}
being independent of $n$.

The eigenvector equation \eqref{eigenvector equation} can be written in the vector form:
\begin{equation}\label{}
    \vec u_n:=
    \left(%
    \begin{array}{c}
      u_{n-1} \\
      u_n \\
    \end{array}%
    \right),
    \quad
    B_n(\lambda):=
    \left(%
    \begin{array}{cc}
      0 & 1 \\
      -\frac{a_{n-1}}{a_n} & \frac{\lambda-b_n}{a_n} \\
    \end{array}%
    \right),
\end{equation}
\begin{equation}
    \vec u_{n+1}=B_n(\lambda)\vec u_n
    ,\quad n\geqslant 2.
\end{equation}
The matrix $B_n$ is called the {\it transfer-matrix} for the equation \eqref{eigenvector equation}.

In our case, \eqref{Jacobi matrix}--\eqref{p-q}, the transfer-matrix for every $\lambda$ has the limit
\begin{equation}\label{limit matrix}
    \left(%
    \begin{array}{cc}
      0 & 1 \\
      -1 & 2 \\
    \end{array}%
    \right).
\end{equation}
The eigenvalues of this limit matrix coincide, and, by analogy to the case of constant coefficients, when the roots of the characteristic equation are the same as the eigenvalues of the transfer-matrix, this is called the {\it double root case}, or {\it critical case}. The matrix \eqref{limit matrix} is not proportional to the identity, so it is not diagonalizable and is similar to a Jordan box. For this reason our situation can be also called the {\it Jordan box case}. Asymptotic analysis of solutions can get involved in this case, and several papers were devoted to studying examples of such Jacobi matrices, among them \cite{Damanik-Naboko-2007,Janas-2006,Janas-Naboko-Sheronova-2009,Naboko-Simonov-2010}.

\section{The stabilized matrix}\label{section stabilized}

In this section let $\mathcal J$ be a Jacobi matrix
\begin{equation}\label{general Jacobi matrix}
    \mathcal J=
    \left(%
    \begin{array}{ccccc}
    b_1 & a_1 & 0 & 0 & \cdots \\
    a_1 & b_2 & a_2 & 0 &\cdots \\
    0 & a_2 & b_3 & a_3 &\cdots \\
    0 & 0 & a_3 & b_4 & \cdots \\
    \vdots & \vdots & \vdots & \vdots & \ddots \\
    \end{array}%
    \right),
\end{equation}
with arbitrary sequences $\{a_n\}_{n=1}^{\infty},\{b_n\}_{n=1}^{\infty}$ of positive and real numbers, respectively. Consider the bounded matrix $\mathcal J_N$ which has the sequence $a_1,a_2,...,a_{N-1},a_N,a_N,a_N,...$ of off-diagonal entries and the sequence $b_1,b_2,...,b_{N-1},b_N,b_N,b_N,...$ on the main diagonal:
\begin{equation}\label{J-N}
    \mathcal J_N=
    \left(%
    \begin{array}{ccccccc}
    b_1 & a_1 & 0 & \cdots & 0 & 0 & \vdots \\
    a_1 & b_2 & a_2 & \cdots & 0 &0 & \vdots \\
    0 & a_2 & \ddots & \ddots & 0 & 0 & \vdots \\
    \vdots & \vdots & \ddots & b_N & a_N & 0 & \vdots \\
    0 & 0 & 0 & a_N & b_N & a_N & \vdots \\
    0 & 0 & 0 & 0 & a_N & b_N & \ddots \\
    \vdots & \vdots & \vdots & \vdots & \vdots & \ddots & \ddots \\
    \end{array}%
    \right).
\end{equation}
The matrix $\mathcal J_N$ is a scaled and shifted discrete Schr\"odinger operator perturbed by finitely supported sequences of diagonal and off-diagonal entries (finite rank perturbation). It is known that its spectrum is purely absolutely continuous on the interval $[b_N-2a_N,b_N+2a_N]$, which follows, e.g., from the subordinacy theory \cite{Khan-Pearson-1992}, and discrete and finite on the rest of the real line. One can write an explicit formula for its spectral density in terms of the orthogonal polynomials for the original matrix $\mathcal J$.

    \begin{proposition}\label{prop density of J-N}
     Consider a Jacobi matrix $\mathcal J$ given by \eqref{general Jacobi matrix} with some sequences $\{a_n\}_{n=1}^{\infty}$ of positive numbers and $\{b_n\}_{n=1}^{\infty}$ of real numbers. Let  $\mathcal J_N$ be defined by \eqref{J-N}. Then its spectral density is
    \begin{equation}\label{density of J-N}
        \rho'_N(\lambda)=\frac{\sqrt{1-\left(\frac{\lambda-b_N}{2a_N}\right)^2}}{\pi a_N|P_{N+1}(\lambda)-z_N(\lambda)P_N(\lambda)|^2},\quad\lambda\in(b_N-2a_N,b_N+2a_N),
    \end{equation}
    where $\{P_n(\lambda)\}_{n=1}^{\infty}$ are the orthogonal polynomials of the first kind associated with the matrix $\mathcal J$ and
    \begin{equation}\label{z-N}
        z_N(\lambda)=\frac{\lambda-b_N}{2a_N}- i\sqrt{1-\left(\frac{\lambda-b_N}{2a_N}\right)^2}
    \end{equation}
    is the boundary value of the analytic branch such that $|z_N(\lambda)|<1$ for $\lambda\in\mathbb C\backslash[b_N-2a_N,b_N+2a_N]$.
    \end{proposition}

This result is analogous to the classical Titchmarsh--Weyl formula for the differential Schr\"o\-din\-ger operator on the half-line with summable potential and is more or less well-known. A version of it is contained in \cite{Aptekarev-Geronimo-2016}, but because of different numbering of entries the stabilized matrix is defined slightly differently there, and hence we cannot literally use that formulation. Proposition \ref{prop density of J-N} therefore needs a separate proof which we provide in Appendix A. A much more general version, for the sequences of $\{a_n\}_{n=1}^{\infty}$ and $\{b_n\}_{n=1}^{\infty}$ of bounded variation, can be found in \cite{Mate-Nevai-Totik-1985}, see also \cite{Van Assche-Geronimo-1988}.

Stabilized matrices approximate the original matrix as $N\to\infty$, so knowing the spectral density of $\mathcal J_N$ we can pass to the limit and find the spectral density of $\mathcal J$. The next two propositions specify the exact sense of this limit passage, both of them are more or less standard. We use the following notation: for $-\infty\leqslant A<B\leqslant+\infty$
\begin{equation}\label{}
    C_{c}(A,B)=\{\varphi\in C(A,B)\ |\ {\rm supp\,}\varphi\text{ is compact}\},
\end{equation}
\begin{equation}\label{}
    C_{0}(A,B)=\{\varphi\in C(A,B)\ |\ \forall\varepsilon\ \exists\text{ compact }K\subset(A,B):\ |\varphi(x)|<\varepsilon,\ \forall x\in(A,B)\backslash K\}.
\end{equation}
$C_{0}(A,B)$ is a Banach space with the norm $\|\varphi\|=\sup_{x\in\mathbb R}|\varphi(x)|$, $C_{c}(A,B)$ is its dense linear subset, the space $C_{0}^*(A,B)$ consists of finite complex-valued Borel measures (automatically regular), \cite[Ch. 3,6]{Rudin-1987}. The following proposition can be found, for example, in \cite{Aptekarev-Geronimo-2016}.

    \begin{proposition}\label{prop weak convergence}
    Let $\mathcal J$ be a Jacobi matrix \eqref{general Jacobi matrix} with some sequences $\{a_n\}_{n=1}^{\infty}$ of positive numbers, $\{b_n\}_{n=1}^{\infty}$ of real numbers, such that $\mathcal J$ is in the limit-point case, and $\mathcal J_N$ be defined by \eqref{J-N}. Then $\rho_N\to\rho$ in the $*$-weak sense as $N\to\infty$.
    \end{proposition}

 The next elementary proposition is also essentially well-known, but it is convenient for us to use it in the following special form.

    \begin{proposition}\label{prop weak limit}
    Let $\mathcal J$ be a Jacobi matrix \eqref{general Jacobi matrix} with some sequences $\{a_n\}_{n=1}^{\infty}$ of positive numbers, $\{b_n\}_{n=1}^{\infty}$ of real numbers, such that $\mathcal J$ is in the limit-point case, and $\mathcal J_N$ be defined by \eqref{J-N}, $-\infty\leqslant A<B\leqslant+\infty$. If there exists an increasing sequence $\{N_k\}_{k=1}^{\infty}$ such that $\rho'_{N_k}(\lambda)\to f(\lambda)$ as $k\to\infty$ uniformly in $\lambda\in K$ for every fixed compact set $K\subset(A,B)$, then the spectral measure $\rho$ of the operator $\mathcal J$ is absolutely continuous on the interval $(A,B)$ and $\rho'(\lambda)=f(\lambda)$ for a.a. $\lambda\in(A,B)$.
    \end{proposition}

\begin{proof}
By Proposition \ref{prop weak convergence}, $\rho_{N_k}\to\rho$ $*$-weakly as $k\to\infty$ in $C_0^*(\mathbb R)$. Hence $\rho_{N_k}|_{(A,B)}\to\rho|_{(A,B)}$ $*$-weakly in $C_0^*(A,B)$: for any $\varphi\in C_0(A,B)$ consider its continuation to $\mathbb R$ by zero, $\widetilde\varphi\in C_0(\mathbb R)$, and convergence follows. Since one sequence in this topology cannot have two limits, it is enough to prove that $d\rho_{N_k}(\lambda)|_{(A,B)}\to f(\lambda)d\lambda$ $*$-weakly in $C_0^*(A,B)$. For every $\varphi\in C_{c}(A,B)$ we have:
\begin{multline}\label{}
    \left|\int_A^B\rho_{N_k}'(\lambda)\varphi(\lambda)d\lambda-\int_A^Bf(\lambda)\varphi(\lambda)d\lambda\right|
    =\left|\int_A^B(\rho_{N_k}'(\lambda)-f(\lambda))\varphi(\lambda)d\lambda\right|
    \\
    \leqslant\sup_{x\in{\rm supp\,}\varphi}|\rho_{N_k}'(x)-f(x)|\int_A^B|\varphi(\lambda)|d\lambda\to0,\quad k\to\infty.
\end{multline}
By the Banach--Steinhaus theorem, owing to the uniform estimate
\begin{equation}\label{}
    \|\rho'_{N_k}(\lambda)d\lambda\|_{C_0^*(A,B)}\leqslant 1,
\end{equation}
convergence holds for every $\varphi\in C_0(A,B)$ as well. Thus $d\rho(\lambda)|_{(A,B)}=f(\lambda)d\lambda$, which completes the proof.
\end{proof}

\section{Spectral density in the critical case}\label{section main}

In this section we formulate and prove the main result of the paper. We investigate the asymptotics of solutions to \eqref{eigenvector equation} as $n\to\infty$ locally in $\lambda$. Let us fix some $0<r<R<\infty$ and consider the open set
\begin{equation}\label{Omega-0}
    \Omega_0:=\{\lambda\in\mathbb C:r<|\lambda|<R\}\backslash\mathbb R_-,
\end{equation}
see Figure \ref{fig 1}. $\overline\Omega_0$ denotes the closure containing both sides (considered to be different) of the cut along $\mathbb R_-$ (closure on the Riemannian surface for $\sqrt\lambda$). Writing $[-R,-r]$ we will always mean the upper side of the cut.
\begin{figure}[h]
\begin{center}
\includegraphics[width=0.4\textwidth]{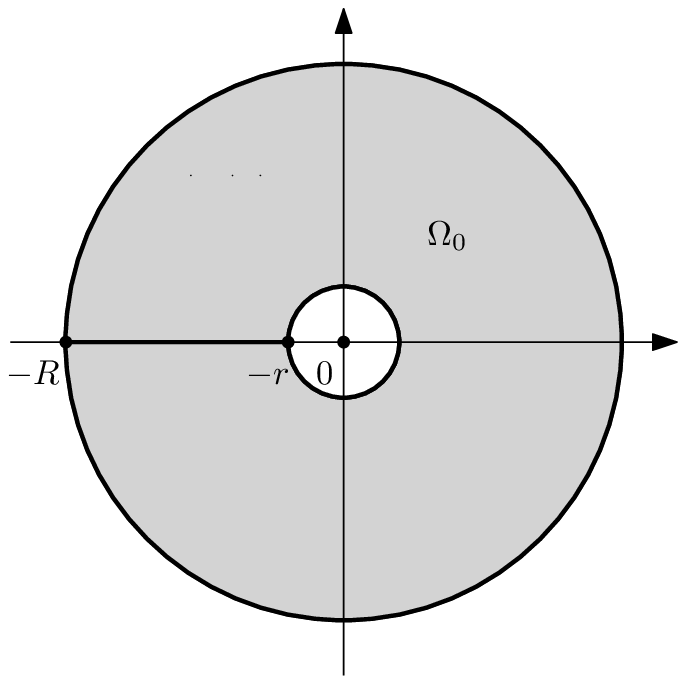}
\caption{The domain $\Omega_0$} \label{fig 1}
\end{center}
\end{figure}

    \begin{theorem}\label{main theorem}
    Let the the entries of the Jacobi matrix $\mathcal J$ be given by
    \begin{equation}\label{entries thm}
        a_n=n^{\alpha}+p_n,\quad b_n=-2n^{\alpha}+q_n
    \end{equation}
    with
    \begin{equation}\label{alpha}
        \alpha\in(0,1)
    \end{equation}
    and real sequences $\{p_n\}_{n=1}^{\infty}$ and $\{q_n\}_{n=1}^{\infty}$ such that $a_n>0$ for all $n$ and
    \begin{equation}\label{p-q from th}
        \left\{\frac{p_n}{n^{\frac{\alpha}2}}\right\}_{n=1}^{\infty},\left\{\frac{q_n}{n^{\frac{\alpha}2}}\right\}_{n=1}^{\infty}\in l^1.
    \end{equation}
    Consider the domain $\Omega_0$. There exists $N_0\in\mathbb N$ such that for every $\lambda\in\overline\Omega_0$ the equation
    \begin{equation}\label{eigenvector equation thm}
        a_{n-1}u_{n-1}+b_nu_n+a_nu_{n+1}=\lambda u_n,\quad n\geqslant 2,
    \end{equation}
    has a solution $u^-(\lambda):=\{u_n^-(\lambda)\}_{n=1}^{\infty}$ with the asymptotics
    \begin{equation}\label{u-pm}
        u_n^{-}(\lambda)=\left(\prod_{l=N_0}^n\eta_l^{-}(\lambda)\right)(1+o(1)),\quad n\to\infty,
    \end{equation}
    uniform in $\lambda\in\overline\Omega_0$, where\footnote{the branch of the square roots  should be chosen so that they are positive for positive $\lambda$}
    \begin{equation}\label{eta-}
        \eta^{-}_n(\lambda)
        =
        1+\frac{\lambda}{2n^{\alpha}}+\frac{\alpha}{4n}
        -
        \frac{\sqrt{\lambda}}{n^{\frac{\alpha}2}}
        \sqrt{1+\frac{\lambda}{4n^{\alpha}}+\frac{\alpha}{\lambda n^{1-\alpha}}},
    \end{equation}
    the value of $N_0$ is chosen so that $\eta^-_n(\lambda)\neq0$ for $\lambda\in\overline\Omega_0$, $n\geqslant N_0$. For every $n$ $u^-_n$ is analytic in $\Omega_0$ and continuous in
    $\overline\Omega_0$. For $\lambda\in[-R,-r]$ there exists a non-zero limit
    \begin{equation}\label{H}
        H(\lambda ):=\lim_{n\to\infty}n^{\frac{\alpha}4}\prod_{l=N_0}^n|\eta^-_l(\lambda)|,
    \end{equation}
    which is continuous in $\lambda$. For $\lambda\in[-R,-r]$ the sequence $u^+(\lambda):=\{u^+_n(\lambda)\}_{n=1}^{\infty}$ with $u^+_n(\lambda)=\overline{u^-_n(\lambda)}$ is another solution of the equation \eqref{eigenvector equation thm} and the nonzero Wronskian
    \begin{equation}\label{Wronskian u+u-}
        W\{u^+(\lambda),u^-(\lambda)\}=-2i\sqrt{-\lambda}H^2(\lambda).
    \end{equation}
    The orthogonal polynomials of the first kind associated with $\mathcal J$ can be expressed as
    \begin{equation}\label{decomposition}
        P_n(\lambda)=\Psi(\lambda)u_n^{+}(\lambda)+\overline{\Psi(\lambda)}u_n^{-}(\lambda),
        \quad\lambda\in[-R,-r],\quad n\in\mathbb N,
    \end{equation}
    where
    \begin{equation}\label{Psi}
        \Psi(\lambda)=\frac{u_0^-(\lambda)}{2i\sqrt{-\lambda}H^2(\lambda)},
    \end{equation}
    $u_0^-(\lambda):=(\lambda-b_1)u_1^-(\lambda)-a_1u_2^-(\lambda)$ (assuming formally in \eqref{eigenvector equation thm} that $a_0:=1$). The spectral density of $\mathcal J$ is given by the formula
    \begin{equation}\label{density}
        \rho'(\lambda)=\frac1{4\pi\sqrt{-\lambda}|\Psi(\lambda)|^2H^2(\lambda)}=\frac{\sqrt{-\lambda}H^2(\lambda)}{\pi|u^-_0(\lambda)|^2},
        \quad\lambda\in[-R,-r].
    \end{equation}
    \end{theorem}

    \begin{remark}
    Another critical case with $b_n=2n^{\alpha}+q_n$ can be easily reduced to the situation \eqref{entries thm} with $\lambda$ replaced by $-\lambda$.
    \end{remark}

    \begin{remark}
    The definition of the solution $u^-_n(\lambda)$  depends through the value of $N_0$ on the set $\Omega_0$ (i.e., on $r$ and $R$), as well as the coefficients $\Psi$, $H$, $u^-_0$: they differ by multiplication by a product of some of the values $\eta^-_n(\lambda)$ for ``small'' $n$, a function of $\lambda$. Unfortunately, this function cannot be taken the same for the whole set $(\overline{\mathbb C\backslash\mathbb R_-})\backslash\{0\}$. At the same time, the functions $|\Psi(\lambda)|H(\lambda)$, $\frac{H(\lambda)}{|u^-_0(\lambda)|}$ and hence $\rho'(\lambda)$ do not depend on $r$ and $R$. Finally, we can take any compact set in $(\overline{\mathbb C\backslash\mathbb R_-})\backslash\{0\}$ instead of $\overline\Omega_0$.
    \end{remark}

    \begin{remark}\label{rem on the form of asymptotics}
    1.
    In the particular case $\alpha\in\left(\frac12,\frac23\right)$ considered in \cite{Janas-Naboko-Sheronova-2009} the formula \eqref{u-pm} and its conjugate version can be written as
    \begin{equation}\label{JNS}
        u_n^{\pm}(\lambda)=\frac{H(\lambda)+o(1)}{n^{\frac{\alpha}4}}
        \exp\left(\pm i\left(
        \frac{\sqrt{-\lambda}n^{1-\frac{\alpha}2}}{1-\frac{\alpha}2}
        -\frac{n^{\frac{\alpha}2}}{\sqrt{-\lambda}}
        +\frac{(\sqrt{-\lambda})^3n^{1-\frac{3\alpha}2}}{24\left(1-\frac{3\alpha}2\right)}
        +\varphi_0(\lambda)\right)\right).
    \end{equation}
    where $\varphi_0$ is some real-valued function. The formula \eqref{JNS} coincides up to a constant in $n$ multiple with the formula from the work \cite{Janas-Naboko-Sheronova-2009}. Note that the restriction $\alpha\in\left(\frac12,\frac23\right)$ is essential here.

    2.
    It is easy to see that in the case $\alpha\in(0,1)\backslash(\frac12,\frac23)$ the asymptotics of $\prod_{l=1}^n\eta^-_l$ has the form similar to \eqref{JNS}. The power terms in $n$ in the exponent have orders from the interval $(0,1)$ and correspond to non-summable terms in the asymptotic expansions of $\ln \eta^-_n$ as $n\to\infty$. The number of such terms depends on $\alpha$ and grows infinitely as $\alpha$ approaches $0$ or $1$. At the same time the decay of polynomials is always of the order $\frac1{n^{\frac{\alpha}4}}$ for $\lambda<0$, and the asymptotics of the solution $u^-$ can be written as
    \begin{equation}\label{}
        u^{\pm}_n(\lambda)=\frac{H(\lambda)+o(1)}{n^{\frac{\alpha}4}}\exp(\pm i\phi_n(\lambda)),\quad n\to\infty,
    \end{equation}
    where $\phi_n(\lambda)=\sum\limits_{k=0}^{K}\varphi_k(\lambda)n^{\alpha_k}$ with some $K$, $0=\alpha_0<\alpha_1<\alpha_2<...<\alpha_K<1$ and some real-valued $\{\varphi_k(\lambda)\}_{k=1}^{K}$. Elementary calculations show that the number $K\sim\left(\frac2{\alpha}+\frac1{2(1-\alpha)}\right)$ as $\alpha\to0$ or $\alpha\to1$.
    \end{remark}

\begin{proof}
Consider $\lambda\in\overline\Omega_0$. For the system
\begin{equation}\label{system u}
    \vec u_{n+1}=B_n(\lambda)\vec u_n,\quad n\geqslant2,
\end{equation}
which is equivalent to the eigenvector equation, we look for a sequence $\{S_n(\lambda)\}_{n=N_0}^{\infty}$ of diagonal matrices
\begin{equation}\label{}
    S_n(\lambda)=
    \begin{pmatrix}
    s_n^+(\lambda) & 0 \\
    0 & s_n^-(\lambda) \\
    \end{pmatrix}
\end{equation}
such that the transformation
\begin{equation}\label{transformation u to v}
    \vec u_n=S_n(\lambda)\vec v_n,\quad \vec v_{n+1}=S_{n+1}^{-1}(\lambda)B_n(\lambda)S_n(\lambda)\vec v_n,\quad n\geqslant N_0,
\end{equation}
leads to the system (reminding the system for the discrete Schr\"odinger operator with the spectral parameter on the boundary of the essential spectrum)
\begin{equation}\label{system v}
    \vec v_{n+1}=
    \begin{pmatrix}
      0 & 1 \\
      -1+c_n(\lambda) & 2 \\
    \end{pmatrix}\vec v_n,
    \quad n\geqslant N_0,
\end{equation}
with some real sequence $\{c_n(\lambda)\}_{n=N_0}^{\infty}$.  The value $N_0\in\mathbb N$ will be chosen large enough and uniform in $\lambda\in\overline\Omega_0$ here and everywhere in the paper. This form corresponds by
$\vec v_n=\begin{pmatrix}
x_n \\
 x_{n+1}
 \end{pmatrix}$ to the three-term recurrence relation $x_{n+2}-2x_{n+1}+(1-c_n(\lambda))x_n=0$, which was studied in the work of R.-J. Kooman, \cite{Kooman-2007}. In order to obtain this we need the following equalities to hold:
\begin{equation}\label{equality for s}
    \begin{pmatrix}
     0 & \frac{s_n^-(\lambda)}{s_{n+1}^+(\lambda)} \\
     -\frac{s_n^+(\lambda)a_{n-1}}{s_{n+1}^-(\lambda)a_n} & \frac{s_n^-(\lambda)(\lambda-b_n)}{s_{n+1}^-(\lambda)a_n} \\
    \end{pmatrix}
    =
    \begin{pmatrix}
    0 & 1 \\
    -1+c_n(\lambda) & 2 \\
    \end{pmatrix}
    ,\quad n\geqslant N_0.
\end{equation}
From the right column $\frac{s_n^-(\lambda)(\lambda-b_n)}{s_{n+1}^-(\lambda)a_n}=2$, $s_n^-(\lambda)=s_{n+1}^+(\lambda)$. Denote
\begin{equation}\label{}
    d_n(\lambda):=\frac{\lambda-b_n}{2a_n},\quad n\geqslant1,
\end{equation}
then $\frac{s^-_{n+1}(\lambda)}{s^-_n(\lambda)}=d_n(\lambda)$.  The index $N_0$ will be chosen large enough to ensure, in particular, that $d_n(\lambda)\neq0$ for $\lambda\in\overline\Omega_0$ for $n\geqslant N_0-1$. Take
\begin{equation}\label{}
    s_n^-(\lambda):=\prod_{l=N_0-1}^{n-1}d_l(\lambda),\quad s_n^+(\lambda):=\prod_{l=N_0-1}^{n-2}d_l(\lambda),\quad n\geqslant N_0,
\end{equation}
so that
\begin{equation}\label{S-n}
    S_n(\lambda)=\left(\prod_{l=N_0-1}^{n-2}d_l(\lambda)\right)
    \begin{pmatrix}
      1 & 0 \\
      0 & d_{n-1}(\lambda) \\
    \end{pmatrix},
    \quad n\geqslant N_0.
\end{equation}
From the equality in the lower-left entries in \eqref{equality for s} we have
\begin{equation}\label{}
    -1+c_n(\lambda)=-\frac{s_n^+(\lambda)a_{n-1}}{s_{n+1}^-(\lambda)a_n},
\end{equation}
therefore one has to define
\begin{equation}\label{}
    c_n(\lambda):=1-\frac{4a_{n-1}^2}{(\lambda-b_{n-1})(\lambda-b_n)},\quad n\geqslant N_0.
\end{equation}
 Note that $c_n(\lambda)\to0$ as $n\to\infty$ in the critical case. The system \eqref{system v} by the substitution
\begin{equation}\label{transformation v to w}
    \vec v_n=
    \begin{pmatrix}
      1 & 0 \\
      1 & 1 \\
    \end{pmatrix}
    \vec w_n
\end{equation}
is further transformed to the system
\begin{equation}\label{system w}
    \vec w_{n+1}=
    \begin{pmatrix}
      1 & 1 \\
      c_n(\lambda) & 1 \\
    \end{pmatrix}
    \vec w_n,\quad n\geqslant N_0,
\end{equation}
which has the same form as in \cite{Kooman-2007}. Following the method of Kooman we look for sequences $\{g_n^{\pm}(\lambda)\}_{n=N_0}^{\infty}$ such that the substitution
\begin{equation}\label{transformation w to y}
\vec w_n=
    \begin{pmatrix}
    1 & 1 \\
    g_n^+(\lambda) & g_n^-(\lambda) \\
    \end{pmatrix}
    \vec y_n,
\end{equation}
transforms the system \eqref{system w} to
\begin{multline}\label{system y}
    \vec y_{n+1}
    =
    \Bigg(
    \begin{pmatrix}
      1+g_n^+ & 0 \\
      0 & 1+g_n^- \\
    \end{pmatrix}
    +
    \frac1{g_{n+1}^--g_{n+1}^+}
    \\
    \times
    \begin{pmatrix}
      g_{n+1}^+-g_n^++g_n^+g_{n+1}^+-c_n & g_{n+1}^--g_n^-+g_n^-g_{n+1}^--c_n \\
      -(g_{n+1}^+-g_n^++g_n^+g_{n+1}^+-c_n) & -(g_{n+1}^--g_n^-+g_n^-g_{n+1}^--c_n ) \\
    \end{pmatrix}
    \Bigg)
    \vec y_n,
\end{multline}
which has the Levinson (L-diagonal) form \cite{Coddington-Levinson-1955,Bodine-Lutz-2015}, if the second term in the coefficient matrix in \eqref{system y} is summable. Combining the substitutions \eqref{transformation u to v}, \eqref{transformation v to w} and \eqref{transformation w to y} we see that solutions of the systems \eqref{system u} and \eqref{system y} are related by the equality
\begin{equation}\label{}
    \vec u_n=S_n(\lambda)
    \begin{pmatrix}
    1 & 1 \\
    1+g^+_n(\lambda) & 1+g^-_n(\lambda) \\
    \end{pmatrix}
    \vec y_n,\quad n\geqslant N_0.
\end{equation}
Consider the sequence $\{c_n(\lambda)\}_{n=N_0}^{\infty}$. It has the asymptotics
\begin{multline}\label{c n general}
    c_n(\lambda)=1-\frac{4((n-1)^{\alpha}+p_{n-1})^2}{(\lambda+2(n-1)^{\alpha}-q_{n-1})(\lambda+2n^{\alpha}-q_n)}
    \\
    =\frac{\frac{\lambda}{n^{\alpha}}+\frac{\lambda^2}{4n^{2\alpha}}+\frac{\alpha}n}{1+\frac{\lambda}{n^{\alpha}}+\frac{\lambda^2}
    {4n^{2\alpha}}}+\frac{r^{(1)}_n(\lambda)}{n^{\frac{\alpha}2}}=\chi_n(\lambda)+\frac{r^{(1)}_n(\lambda)}{n^{\frac{\alpha}2}},
\end{multline}
where $\{\sup_{\lambda\in\overline\Omega_0}\{|r^{(1)}_n(\lambda)|\}\}_{n=N_0}^{\infty}\in l^1$ for sufficiently large $N_0$. Here we use the notation
\begin{equation}\label{chi}
        \chi_n(\lambda):=
        \frac
        {\frac{\lambda}{n^{\alpha}}+\frac{\lambda^2}{4n^{2\alpha}}+\frac{\alpha}n}
        {\left(1+\frac{\lambda}{2n^{\alpha}}\right)^2}.
\end{equation}
The index $N_0$ will be chosen so that the roots and the poles of $\chi_n$ and the poles of $c_n$ lie outside of $\overline\Omega_0$ for $n\geqslant N_0$.

An important property which also arises in \cite[Theorem 1, case 1]{Kooman-2007} as a condition, and which one can check straightforwardly here, is that
\begin{equation}\label{condition Kooman}
    \left\{\sup_{\lambda\in \overline\Omega_0}\left|\frac{\chi_{n+1}(\lambda)-\chi_n(\lambda)}{\chi_n(\lambda)}+\frac{\alpha}n\right|\right\}_{n=N_0}^{\infty}\in l^1,
\end{equation}
due to the property
\begin{equation}\label{}
    \chi_{n+1}(\lambda)=\chi_n(\lambda)\left(1-\frac{\alpha}n+r_n^{(2)}(\lambda)\right)
    \quad\text{with}\quad\left\{\sup_{\lambda\in \overline\Omega_0}|r^{(2)}_n(\lambda)|\right\}_{n=N_0}^{\infty}\in l^1
\end{equation}
for sufficiently large $N_0$. Let us define
\begin{equation}\label{g}
    g_n^{\pm}(\lambda):=\pm\sqrt{\chi_n(\lambda)}+\frac{\alpha}{4n},\quad n\geqslant N_0.
\end{equation}
To specify the branch of the square root note that the function $\chi_n$ has two roots $\lambda^{\pm}_n:=2n^{\alpha}(-1\pm\sqrt{1-\frac{\alpha}n})$ such that $\lambda^-_n\to-\infty$, $\lambda^+_n\to0$ as $n\to\infty$, and one pole $\mu_n:=-2n^{\alpha}$. We choose $N_0$ so large that $\lambda^-_n,\mu_n<-3R$ and $\lambda^+_n>-r$ for all $n\geqslant N_0$. The branch can be chosen in $\mathbb C\backslash[\lambda^-_n,\lambda^+_n]$ by specifying that $\sqrt{\chi_n(0)}=\sqrt{\frac{\alpha}n}$ is a positive number (the principle branch). For every $n\geqslant N_0$ the function $\sqrt{\chi_n(\lambda)}$ is analytic in $\Omega_0$ and continuous in $\overline{\Omega}_0$. The same applies to the function $\sqrt{\frac{\lambda}{n^{\alpha}}+\frac{\lambda^2}{4n^{2\alpha}}+\frac{\alpha}n}$, and we can write in $\overline\Omega_0$ that
\begin{equation}\label{}
    \sqrt{\chi_n(\lambda)}
    =
    \frac{\sqrt{\frac{\lambda}{n^{\alpha}}+\frac{\lambda^2}{4n^{2\alpha}}+\frac{\alpha}n}}{1+\frac{\lambda}{2n^{\alpha}}}
    =
    \frac{\frac{\sqrt{\lambda}}{n^{\frac{\alpha}2}}
        \sqrt{1+\frac{\lambda}{4n^{\alpha}}+\frac{\alpha}{\lambda n^{1-\alpha}}}}{1+\frac{\lambda}{2n^{\alpha}}},\quad n\geqslant N_0
\end{equation}
which corresponds to the choice of branch in \eqref{eta-}. Taking into account the arguments of the numerator and the denominator of the second expression in $\overline\Omega_0\cap\overline{\mathbb C}_{\pm}$, one can see (using elementary geometric considerations) that
\begin{equation}\label{Re chi}
    {\rm Re\,}\sqrt{\chi_n(\lambda)}\geqslant0\text{ for }\lambda\in\overline\Omega_0,n\geqslant N_0
\end{equation}
for $N_0$ sufficiently large. Then from \eqref{condition Kooman} it follows that
\begin{equation}\label{condition summability}
    \left\{\sup_{\lambda\in \overline\Omega_0}\left|\frac{g_{n+1}^{\pm}(\lambda)-g_n^{\pm}(\lambda)+g_n^{\pm}(\lambda)g_{n+1}^{\pm}(\lambda)-\chi_n(\lambda)}{g_{n+1}^-(\lambda)-g_{n+1}^+(\lambda)}\right|\right\}_{n=N_0}^{\infty}\in l^1.
\end{equation}
Indeed, one can easily check that
\begin{multline}\label{}
    g^{\pm}_{n+1}(\lambda)=\pm\sqrt{\chi_n(\lambda)}\sqrt{1-\frac{\alpha}n+r^{(3)}_n(\lambda)}+\frac{\alpha}{4(n+1)}
    \\
    =\pm\sqrt{\chi_n(\lambda)}\mp\frac{\alpha\sqrt{\chi_n(\lambda)}}{2n}+\frac{\alpha}{4n}+\sqrt{\chi_n(\lambda)}r^{(4)}_n(\lambda),
\end{multline}
\begin{equation}\label{}
    g^{\pm}_{n+1}(\lambda)-g^{\pm}_n(\lambda)=\sqrt{\chi_n(\lambda)}\left(\mp\frac{\alpha}{2n}+r^{(5)}_n(\lambda)\right),
\end{equation}
\begin{equation}\label{equality g- - g+}
    g^-_{n+1}(\lambda)-g^+_{n+1}(\lambda)=-2\sqrt{\chi_{n+1}(\lambda)}=\sqrt{\chi_n(\lambda)}\left(-2+\frac{\alpha}n+r^{(5)}_n(\lambda)\right),
\end{equation}
\begin{equation}\label{}
    g^{\pm}_n(\lambda)g^{\pm}_{n+1}(\lambda)=\chi_n(\lambda)\pm\frac{\alpha\sqrt{\chi_n(\lambda)}}{2n}+\sqrt{\chi_n(\lambda)}r^{(6)}(\lambda)
\end{equation}
with $\{\sup_{\lambda\in \overline\Omega_0}\{|r^{(3)}_n(\lambda)|,|r^{(4)}_n(\lambda)|,|r^{(5)}_n(\lambda)|,|r^{(6)}_n(\lambda)|\}\}_{n=N_0}^{\infty}\in l^1$ for sufficiently large $N_0$, from which \eqref{condition summability} follows. From \eqref{c n general}, \eqref{equality g- - g+} and \eqref{chi} we have
\begin{equation}\label{}
    \left\{\sup_{\lambda\in \overline\Omega_0}\left|\frac{c_n(\lambda)-\chi_n(\lambda)}{g_{n+1}^-(\lambda)-g_{n+1}^+(\lambda)}\right|\right\}_{n=N_0}^{\infty}\in l^1.
\end{equation}
Therefore observing the second term of the system \eqref{system y} we can write it in the form
\begin{equation}\label{system y with R}
    \vec y_{n+1}=
     \left(
    \begin{pmatrix}
      1+g_n^+(\lambda) & 0 \\
      0 & 1+g_n^-(\lambda) \\
    \end{pmatrix}
    +R_n(\lambda)\right)
    \vec y_n, \quad n\geqslant N_0,
\end{equation}
where $\{\sup_{\lambda\in \overline\Omega_0}\|R_n(\lambda)\|\}_{n=N_0}^{\infty}\in l^1$. Clearly, $c_n$ and $g^{\pm}_n$ are analytic in $\Omega$ and continuous in $\overline\Omega_0$ for all $n\geqslant N_0$. Now the tools used to prove the following lemma can be employed to show existence of the solution $y^-_n(\lambda)$ analytic in $\Omega_0$, continuous in $\overline\Omega_0$ (for all $n\geqslant N_0$) and with asymptotics uniform in $\overline\Omega_0$
\begin{equation}\label{solution y-}
    \vec y_n^{\,-}(\lambda)=\left(\prod_{l=N_0}^{n-1}(1+g_l^-(\lambda))\right)(\vec e_-+o(1)),\quad n\to\infty.
\end{equation}

    \begin{lemma}\label{lemma asymptotics}
    Let the sequence $\{\lambda_n\}_{n=1}^{\infty}$ of nonzero complex numbers be such that for some $C>0$ and any $p,q\in\mathbb N$ such that $p\leqslant q$,
    \begin{equation}\label{condition Levinson 1}
        \prod\limits_{l=p}^q|\lambda_l|\geqslant\frac1C.
    \end{equation}
    Let the sequence $\{R_n\}_{n=1}^{\infty}$ of complex $2\times2$ matrices be such that
    \begin{equation*}
        \det\left(
        \begin{pmatrix}
        \lambda_n & 0 \\
        0 & \frac1{\lambda_n} \\
        \end{pmatrix}
        +R_n\right)\neq0,\quad n\geqslant1,
    \end{equation*}
    and
    \begin{equation}\label{condition Levinson 2}
        \sum_{k=1}^{\infty}|\lambda_k|\|R_k\|<\infty.
    \end{equation}
    Then the system
    \begin{equation}\label{system x}
        \vec x_{n+1}=\left(
        \begin{pmatrix}
        \lambda_n & 0 \\
        0 & \frac1{\lambda_n} \\
        \end{pmatrix}
        +R_n\right)\vec x_n,\ n\geqslant1,
    \end{equation}
    has the solution
    \begin{equation}\label{solution x- lemma}
      \vec x^{\,-}_n=\left(\prod_{l=1}^{n-1}\frac1{\lambda_l}\right)(\vec e_-+o(1)),\quad n\to\infty,
    \end{equation}
    where $\vec e_-=
    \begin{pmatrix}
    0
    \\
    1
    \end{pmatrix}$.
    \end{lemma}

\begin{proof}
One can check that solutions of the system \eqref{system x} are the same as solutions of the integral equations
\begin{equation}\label{}
    \vec x_n=
    \begin{pmatrix}
    \prod\limits_{l=1}^{n-1}\lambda_l & 0
    \\
    0 & \prod\limits_{l=1}^{n-1}\frac1{\lambda_l}
    \end{pmatrix}
    \vec f
    -\sum_{k=n}^{\infty}
    \begin{pmatrix}
    \prod\limits_{l=n}^{k}\frac1{\lambda_l} & 0
    \\
    0 & \prod\limits_{l=n}^{k}\lambda_l
    \end{pmatrix}
    R_k\vec x_k
\end{equation}
for arbitrary $\vec f\in\mathbb C^2$ (this is a kind of variation of parameters method). In particular, consider $\vec f=\vec e_-$ and
\begin{equation}\label{equation integral for x-}
    \vec x^{\,-}_n=
    \left(\prod\limits_{l=1}^{n-1}\frac1{\lambda_l}\right)
    \vec e_-
    -\sum_{k=n}^{\infty}
    \begin{pmatrix}
    \prod\limits_{l=n}^{k}\frac1{\lambda_l} & 0
    \\
    0 & \prod\limits_{l=n}^{k}\lambda_l
    \end{pmatrix}
    R_k\vec x^{\,-}_k.
\end{equation}
Then for the new sequence of vectors
\begin{equation}\label{}
  \vec{\tilde x}^{-}_n:=\left(\prod\limits_{l=1}^{n-1}\lambda_l\right)\vec x^{\,-}_n
\end{equation}
the equation \eqref{equation integral for x-} is equivalent to the equation
\begin{equation}\label{equation integral for tilde x-}
    \vec{\tilde x}^{-}_n=
    \vec e_-
    -\sum_{k=n}^{\infty}
    \begin{pmatrix}
    \prod\limits_{l=n}^{k}\frac1{\lambda_l^2} & 0
    \\
    0 & 1
    \end{pmatrix}
    \lambda_kR_k\vec{\tilde x}^{-}_k.
\end{equation}
Denote the family of $2\times2$ matrices
\begin{equation}\label{}
    V_{nk}:=-
    \begin{pmatrix}
    \prod\limits_{l=n}^{k}\frac1{\lambda_l^2} & 0
    \\
    0 & 1
    \end{pmatrix}
    \lambda_kR_k,\quad n,k\geqslant1.
\end{equation}
Since
\begin{equation}\label{estimates Volterra}
  \|V_{nk}\|\leqslant(C^2+1)|\lambda_k|\|R_k\|
\end{equation}
is summable in $k$ by the conditions \eqref{condition Levinson 1} and \eqref{condition Levinson 2}, the equation \eqref{equation integral for tilde x-} can be written as
\begin{equation}\label{}
  \vec{\tilde x}^{-}=\vec e_-+V\vec{\tilde x}^{-}
\end{equation}
with the Volterra operator $V$  in the Banach space $l^{\infty}(\mathbb N;\mathbb C^2)$ defined by the matrix-valued kernel $V_{nk}$. One has
\begin{equation}\label{}
    \|V^j\|\leqslant\frac{\left(\sum_{k=1}^{\infty}(C^2+1)|\lambda_k|\|R_k\|\right)^j}{j!},\quad j\geqslant0,
\end{equation}
\begin{equation}
    \|(I-V)^{-1}\|\leqslant\exp\left(\sum_{k=1}^{\infty}(C^2+1)|\lambda_k|\|R_k\|\right),
\end{equation}
\begin{equation}\label{solution tilde x- as series}
    \vec{\tilde x}^{-}=(I-V)^{-1}\vec e_-=\sum_{j=0}^{\infty}V^j\vec e_-.
\end{equation}
From \eqref{equation integral for tilde x-} and \eqref{estimates Volterra} it follows that $\vec{\tilde x}^{-}_n\to e_-$, $n\to\infty$, and from this we have \eqref{solution x- lemma}, which completes the proof.
\end{proof}

\begin{remark}
This is another variation of the discrete asymptotic Levinson theorem, see \cite[Lemma 2.1]{Benzaid-Lutz-1987}. However, note that the proof of existence of the ``small'' solution does not require the dichotomy condition. This is crucial for uniformity of  asymptotics and for continuity of solutions in the parameter $\lambda$ in what follows. Other formulations of smooth and uniform discrete Levinson theorems, as in \cite{Silva-2007}, do not yield the result we need.
\end{remark}

Consider the following transformation $\vec y_n\to\vec x_n$:
\begin{equation}\label{transformation y to x}
    \vec y_n=\left(\prod_{l=N_0}^{n-1}\sqrt{1+g^+_l(\lambda)}\sqrt{1+g^-_l(\lambda)}\right)\vec x_n,\quad n\geqslant N_0.
\end{equation}
For sufficiently large $N_0$ we have $|g^{\pm}_n(\lambda)|\leqslant\frac12$ for $\lambda\in\overline\Omega$, $n\geqslant N_0$, and we can take the principal branch for both square roots. This substitution transforms \eqref{system y with R} to the system
\begin{equation}\label{}
    \vec x_{n+1}=
    \left(
    \begin{pmatrix}
\frac{\sqrt{1+g^+_n}}{\sqrt{1+g^-_n}} & 0
    \\
    0 &\frac{\sqrt{1+g^-_n}}{\sqrt{1+g^+_n}}
    \end{pmatrix}
    +
    \frac{R_n}{\sqrt{1+g^+_n}\sqrt{1+g^-_n}}
    \right)
    \vec x_n,\quad n\geqslant N_0,
\end{equation}
to which Lemma \ref{lemma asymptotics} is applicable in $\overline\Omega_0$, if $N_0$ is chosen large enough (shifting the index by $N_0-1$ does not change the situation). By the lemma there exists a solution
\begin{equation}\label{solution x-}
    \vec x^{\,-}_n(\lambda)=\left(\prod_{l=N_0}^{n-1}\frac{\sqrt{1+g^+_l(\lambda)}}{\sqrt{1+g^-_l(\lambda)}}\right)(\vec e_-+o(1)),\quad n\to\infty,
\end{equation}
which together with \eqref{transformation y to x} gives \eqref{solution y-}. The constant $C$ from \eqref{condition Levinson 1} can be chosen equal to one for every $\lambda\in\overline\Omega_0$: due to \eqref{Re chi} we have $\left|\frac{1+\sqrt{\chi_n(\lambda)}+\frac{\alpha}{4n}}{1-\sqrt{\chi_n(\lambda)}+\frac{\alpha}{4n}}\right|\geqslant1$ for all $\lambda\in\overline\Omega_0$ and $n\geqslant N_0$. Besides that we have
\begin{equation}\label{}
    \left\{\sup_{\lambda\in \overline\Omega_0}\frac{\|R_n(\lambda)\|}{|1+g^-_n(\lambda)|}\right\}_{n=N_0}^{\infty}\in l^1,
\end{equation}
since $|g^-_n(\lambda)|\leqslant\frac12$ for $\lambda\in\overline\Omega$, $n\geqslant N_0$, which ensures that the estimate of the sum in \eqref{equation integral for tilde x-} provided by \eqref{estimates Volterra} is uniform in $\lambda\in\overline\Omega_0$, and hence the asymptotics in \eqref{solution x-} and \eqref{solution y-} are uniform. Furthermore, the sum in \eqref{solution tilde x- as series} converges absolutely and uniformly (i.e., $\sum_{j=0}^{\infty}\sup_{\lambda\in\overline\Omega_0}\|V^j(\lambda)\vec e_-\|<\infty$), the summands are analytic in $\Omega_0$ and continuous in $\overline\Omega_0$, thus the solutions $\vec x^{\,-}(\lambda)$ and $\vec y^{\,-}(\lambda)$ are also analytic in $\Omega_0$ and continuous in $\overline\Omega_0$. Returning to the system \eqref{system u} with \eqref{transformation w to y}, \eqref{transformation v to w}, \eqref{transformation u to v} we obtain the solution of the system \eqref{system u} analytic in $\Omega_0$ and continuous in $\overline\Omega_0$ with asymptotics as $n\to\infty$ uniform in $\overline\Omega_0$, the new sequence of vectors
\begin{multline}\label{}
    \vec{\tilde u}^{\,-}_n(\lambda):=S_n(\lambda)
    \begin{pmatrix}
    1 & 0 \\
    1 & 1 \\
    \end{pmatrix}
    \begin{pmatrix}
    1 & 1 \\
    g_n^+(\lambda) & g_n^-(\lambda) \\
    \end{pmatrix}
    \vec y_n^{\,-}(\lambda)
    =
    S_n(\lambda)
    \begin{pmatrix}
    1 & 0 \\
    1 & 1 \\
    \end{pmatrix}
    \begin{pmatrix}
    1 & 1 \\
    g_n^+(\lambda) & g_n^-(\lambda) \\
    \end{pmatrix}
    \\
    \times\left(\prod_{l=N_0}^{n-1}\sqrt{1+g^+_l(\lambda)}\sqrt{1+g^-_l(\lambda)}\right)\vec x^{\,-}_n
    =\left(\prod_{l=N_0}^{n-1}h_l^-(\lambda)\right)
    \begin{pmatrix}
    1 & 1 \\
    h_n^+(\lambda) & h_n^-(\lambda) \\
    \end{pmatrix}
    (\vec e_-+o(1)),
\end{multline}
where
\begin{equation}\label{h-n}
    h_n^{\pm}(\lambda):=d_{n-1}(\lambda)(1+g^{\pm}_n(\lambda)).
\end{equation}

As one can see,
\begin{equation}\label{asymptotics d}
    d_n(\lambda)=1+\frac{\lambda}{2n^{\alpha}}+\frac{r^{(7)}_n(\lambda)}{n^{\frac{\alpha}2}}\quad\text{with}\quad\left\{\sup_{\lambda\in \overline\Omega_0}|r^{(7)}_n(\lambda)|\right\}_{n=1}^{\infty}\in l^1.
\end{equation}
Using this formula and the definitions \eqref{g}, \eqref{chi} and \eqref{eta-} we get
\begin{equation}\label{}
    h^-_n(\lambda)=\eta^-_n(\lambda)+r^{(8)}_n(\lambda)\quad\text{with}\quad\left\{\sup_{\lambda\in \overline\Omega_0}|r^{(8)}_n(\lambda)|\right\}_{n=N_0}^{\infty}\in l^1.
\end{equation}
By the choice of $N_0$ we can ensure that for all  for $\lambda\in\overline\Omega$, $n\geqslant N_0$ we have $|\eta^-_n(\lambda)|\geqslant\frac12$. Then for every $n\geqslant N_0$
\begin{equation}\label{}
    \frac{h^-_n(\lambda)}{\eta^-_n(\lambda)}
    =
    1+\frac{r^{(8)}_n(\lambda)}{\eta^-_n(\lambda)}
    =
    1+r^{(9)}_n(\lambda)
    \quad\text{with}\quad\left\{\sup_{\lambda\in \overline\Omega_0}|r^{(9)}_n(\lambda)|\right\}_{n=N_0}^{\infty}\in l^1.
\end{equation}
This quotient is a function analytic in $\Omega_0$ and continuous in $\overline\Omega_0$ without roots. Therefore the product
\begin{equation}\label{}
    \prod_{l=N_0}^{\infty}\frac{h^-_n(\lambda)}{\eta^-_n(\lambda)}=:C_0(\lambda)
\end{equation}
converges and has the same properties. Define the solution proportional to $\vec{\tilde u}^-_n(\lambda)$
\begin{multline}\label{u-n- with matrix}
    \vec u^{\,-}_n(\lambda)=
    \begin{pmatrix}
    u_{n-1}(\lambda)
    \\
    u_n(\lambda)
    \end{pmatrix}
    :=
    \frac{\vec{\tilde u}^-_n(\lambda)}{C_0(\lambda)}
    =
    \left(\prod_{l=N_0}^{n-1}\eta^-_l(\lambda)\right)
    \left(\prod_{l=n}^{\infty}\frac{\eta^-_l(\lambda)}{h_l^-(\lambda)}\right)
    \begin{pmatrix}
    1 & 1 \\
    h_n^+(\lambda) & h_n^-(\lambda) \\
    \end{pmatrix}
    (\vec e_-+o(1))
    \\
    =
    \left(\prod_{l=N_0}^{n-1}\eta^-_l(\lambda)\right)
    \begin{pmatrix}
    1 & 1 \\
    h_n^+(\lambda) & h_n^-(\lambda) \\
    \end{pmatrix}
    (\vec e_-+o(1)),
\end{multline}
from which \eqref{u-pm} follows. The solution $\vec u_n$ is analytic in $\Omega_0$ and continuous in $\overline\Omega_0$ for every $n\geqslant N_0$. Moreover, although this solution is initially defined for $n\geqslant N_0$, it exists for all $n\geqslant2$ (and can be formally defined also for $n=1$ with $a_0:=0$) retaining the same properties, because matrices $B_n(\lambda)$ and $B^{-1}_n(\lambda)$ are entire in $\lambda$ for all $n$.

Note that the form of asymptotics
$\left(\prod_{l=N_0}^{n-1}\eta^-_l(\lambda)\right)
\begin{pmatrix}
1 & 1 \\
h_n^+(\lambda) & h_n^-(\lambda) \\
\end{pmatrix}
(\vec e_-+o(1))$
implies, but is not equivalent to
$\left(\prod_{l=N_0}^{n-1}\eta^-_l(\lambda)\right)
\begin{pmatrix}
1+o(1)\\
1+o(1)
\end{pmatrix}$,
which would be enough to get \eqref{u-pm}. It contains more information which is lost due to degeneracy of the matrix
$\begin{pmatrix}
1&1\\
1&1
\end{pmatrix}
=\lim\limits_{n\to\infty}
\begin{pmatrix}
1 & 1 \\
h_n^+(\lambda) & h_n^-(\lambda) \\
\end{pmatrix}$.
We will employ this more refined form of asymptotics below, in \eqref{calculation Wronskian}.

Throughout this proof we have many times declared that $N_0$ should be chosen sufficiently large so that one one or another property holds true. Evidently, one can choose $N_0$ so that all of them hold at once. This choice is determined only by the values of $r$ and $R$.

For $\lambda\in[-R,-r]$ and large $n$
\begin{equation}\label{}
    |\eta^-_n(\lambda)|^2=\left(1+\frac{\lambda}{2n^{\alpha}}+\frac{\alpha}{4n}\right)^2
    -\frac{\lambda}{n^{\alpha}}-\frac{\lambda^2}{4n^{2\alpha}}-\frac{\alpha}n
    =1-\frac{\alpha}{2n}+\frac{\alpha\lambda}{4n^{1+\alpha}}+\frac{\alpha^2}{16n^2}.
\end{equation}
Thus $n^{\frac{\alpha}2}\prod_{n=N_0}^n|\eta^-_n(\lambda)|^2$ has a finite limit as $n\to\infty$ which is a continuous function of $\lambda$ without roots in $[-R,-r]$. Hence the definition \eqref{H} is correct.

Since the eigenvector equation \eqref{eigenvector equation thm} has real coefficients, the sequence $u^+_n(\lambda):=\overline{u^-_n(\lambda)}$ is its solution for $\lambda\in[-R,-r]$. The sequence
\begin{equation}\label{}
    \vec u^{\,+}_n(\lambda):=
    \begin{pmatrix}
    u^+_{n-1}(\lambda)
    \\
    u^+_n(\lambda)
    \end{pmatrix}
    =\left(\prod_{l=N_0}^{n-1}\overline{\eta^-_l(\lambda)}\right)
    \begin{pmatrix}
    1 & 1
    \\
    \overline{h^+_n(\lambda)} & \overline{h^-_n(\lambda)}
    \end{pmatrix}
    (\vec e_-+o(1))
\end{equation}
is a solution to the system \eqref{system u}. For $\lambda\in[-R,-r]$ and large $n$ one has $h^+_n(\lambda)=\overline{h^-_n(\lambda)}$, hence
\begin{equation}\label{u-n+ with matrix}
    \vec u^{\,+}_n(\lambda)
    =\left(\prod_{l=N_0}^{n-1}\overline{\eta^-_l(\lambda)}\right)
    \begin{pmatrix}
    1 & 1
    \\
    h^-_n(\lambda) & h^+_n(\lambda)
    \end{pmatrix}
    (\vec e_-+o(1))
    =\left(\prod_{l=N_0}^{n-1}\overline{\eta^-_l(\lambda)}\right)
    \begin{pmatrix}
    1 & 1
    \\
    h^+_n(\lambda) & h^-_n(\lambda)
    \end{pmatrix}
    (\vec e_++o(1)),
\end{equation}
where $\vec e_+:=
\begin{pmatrix}
1
\\
0
\end{pmatrix}$. The Wronskian of the solutions $u^+(\lambda)$ and $u^-(\lambda)$ is equal for fixed $\lambda\in[-R,-r]$ for all $n\in\mathbb N$ to
\begin{multline}\label{calculation Wronskian}
    W\{u^+(\lambda),u^-(\lambda)\}=a_n(u^+_n(\lambda)u^-_{n+1}(\lambda)-u^+_{n+1}(\lambda)u^-_n(\lambda))
    =a_n\det
    \begin{pmatrix}
    u^+_n(\lambda) & u^-_n(\lambda) \\
    u^+_{n+1}(\lambda) & u^-_{n+1}(\lambda) \\
    \end{pmatrix}
    \\
    =
    \lim_{n\to\infty}a_n\det
    \left(
    \begin{pmatrix}
    1 & 1 \\
    h_{n+1}^+(\lambda) & h_{n+1}^-(\lambda) \\
    \end{pmatrix}
    (I+o(1))
    \prod_{l=N_0}^n
    \begin{pmatrix}
    \overline{\eta^-_l(\lambda)}  & 0 \\
    0 & \eta^-_l(\lambda)
    \end{pmatrix}
    \right)
    \\
    =\lim_{n\to\infty}a_n(h^-_{n+1}(\lambda)-h^+_{n+1}(\lambda))\left(\prod_{l=N_0}^n|\eta^-_l(\lambda)|^2\right)(1+o(1))
    \\
    =-2\lim_{n\to\infty}a_nd_n(\lambda)\sqrt{\chi_{n+1}(\lambda)}\left(\prod_{l=N_0}^n|\eta^-_l(\lambda)|^2\right)(1+o(1))=-2i\sqrt{-\lambda}H^2(\lambda)
\end{multline}
where we used \eqref{h-n}, \eqref{u-n- with matrix} and \eqref{u-n+ with matrix} to get the third equality, \eqref{h-n} and \eqref{g} to get the fifth and \eqref{entries}, \eqref{asymptotics d}, \eqref{chi} and \eqref{H} for the last equality. Hence $u^+(\lambda)$ and $u^-(\lambda)$ are linearly independent for $\lambda\in
[-R,-r]$ and orthogonal polynomials of the first kind can be expressed as
\begin{equation}\label{asymptotics of P beta}
    P_n(\lambda)=\Psi(\lambda)u_n^+(\lambda)+\overline{\Psi(\lambda)}u_n^-(\lambda)
\end{equation}
with
\begin{equation}
    \Psi(\lambda)=\frac{W\{P(\lambda),u^-(\lambda)\}}{W\{u^+(\lambda),u^-(\lambda)\}}=\frac{a_0(P_0(\lambda)u^-_1(\lambda)-P_1(\lambda)u^-_0(\lambda))}{-2i\sqrt{-\lambda}H^2(\lambda)}=\frac{u^-_0(\lambda)}{2i\sqrt{-\lambda}H^2(\lambda)},
\end{equation}
 where $P(\lambda):=\{P_n(\lambda)\}_{n=1}^{\infty}$ and $P_0(\lambda)$ is the formal solution of the eigenvector equation with $a_0=1$, i.e., $P_0(\lambda)\equiv0$; one can easily check that constancy of the Wronskian can be extended to $\mathbb N\cup\{0\}$. From this we see that $\Psi$ and $u^-_0$ cannot have zeros on $[-R,-r]$.

Using Proposition \ref{prop weak limit} we come to establishing the limit uniform in $[-R,-r]$ of
\begin{equation}\label{rho n}
    \rho'_n(\lambda)=\frac{\sqrt{1-d_n^2(\lambda)}}{\pi a_n|P_{n+1}(\lambda)-z_n(\lambda)P_n(\lambda)|^2}\text{ for a.a. }\lambda\in(b_n-2a_n,b_n+2a_n),
\end{equation}
where
\begin{equation}
    z_n(\lambda)=d_n(\lambda)-i\sqrt{1-d_n^2(\lambda)},
\end{equation}
according to \eqref{density of J-N} and \eqref{z-N}. Firstly, from \eqref{asymptotics d} we have
\begin{equation}\label{}
    \sqrt{1-d_n^2(\lambda)}=\frac{\sqrt{-\lambda}}{n^{\frac{\alpha}2}}\sqrt{1+\frac2{\lambda}n^{\frac{\alpha}2}r^{(7)}_n(\lambda)+o(1)}.
\end{equation}
The term $\frac2{\lambda}n^{\frac{\alpha}2}r^{(7)}_n(\lambda)$ is not vanishing in general. Moreover, it may behave in an irregular way under our condition \eqref{p-q}. However, it can be made vanishing on a proper subsequence. Since we are looking for uniform convergence, this subsequence should not depend on $\lambda$. Let $\varkappa_n:=\max\left\{\frac{|p_n|}{n^{\frac{\alpha}2}},\frac{|q_n|}{n^{\frac{\alpha}2}}\right\}$, $\{\varkappa_n\}_{n=1}^{\infty}\in l^1$ by the condition \eqref{p-q}. One can choose an increasing sequence $\{n_k\}_{k=1}^{\infty}$ such that $\varkappa_{n_k}=o\left(\frac1{n_k^{\alpha}}\right)$, $k\to\infty$. This means that $p_{n_k},q_{n_k}=o\left(\frac1{n_k^{\frac{\alpha}2}}\right)$. Then
\begin{equation}\label{}
    d_{n_k}(\lambda)=\frac{\lambda+2n_k^{\alpha}-q_{n_k}}{2n_k^{\alpha}+2p_{n_k}}=1+\frac{\lambda}{2n_k^{\alpha}}+o\left(\frac1{n_k^{\frac{3\alpha}2}}\right),
\end{equation}
\begin{equation}\label{sqrt 1-d^2}
    \sqrt{1-d_{n_k}^2(\lambda)}=\frac{\sqrt{-\lambda}}{n_k^{\frac{\alpha}2}}\left(1+o\left(\frac1{n_k^{\frac{\alpha}2}}\right)\right),
\end{equation}
\begin{equation}\label{z-n-k}
    z_{n_k}(\lambda)=1-\frac{i\sqrt{-\lambda}}{n_k^{\frac{\alpha}2}}+o\left(\frac1{n_k^{\frac{\alpha}2}}\right),
\end{equation}
and
\begin{equation}\label{h-n-k}
    h^{\pm}_{n_k+1}(\lambda)=1\pm\frac{i\sqrt{-\lambda}}{n_k^{\frac{\alpha}2}}+o\left(\frac1{n_k^{\frac{\alpha}2}}\right).
\end{equation}
The error $o$-terms are uniform in $\lambda\in[-R,-r]$. This also gives $b_{n_k}+2a_{n_k}\to0$ as $k\to\infty$ (while for $b_n-2a_n\to-\infty$ as $n\to\infty$ we do not need a subsequence). Thus we can find $K\in\mathbb N$ such that for $k\geqslant K$ the inclusion $[-R,-r]\subset(b_{n_k}-2a_{n_k},b_{n_k}+2a_{n_k})$ holds and \eqref{rho n} is true for a.a. $\lambda\in[-R,-r]$ and $n=n_k$. Further, on $[-R,-r]$ we have
\begin{equation}\label{}
    P_{n+1}-z_nP_n=\Psi(u^+_{n+1}-z_nu^+_n)+\overline{\Psi}(u^-_{n+1}-z_nu^-_n).
\end{equation}
Using \eqref{u-n+ with matrix} we get
\begin{multline}\label{}
    u^+_{n+1}-z_nu^+_n
    =
    \begin{pmatrix}
    -z_n & 1 \\
    \end{pmatrix}
    \begin{pmatrix}
    u^+_n
    \\
    u^+_{n+1}
    \end{pmatrix}
    \\
    =
    \begin{pmatrix}
    -z_n & 1 \\
    \end{pmatrix}
    \left(\prod_{l=N_0}^n\overline{\eta^-_l}\right)
    \begin{pmatrix}
    1 & 1
    \\
    h^+_{n+1} & h^-_{n+1}
    \end{pmatrix}
    (\vec e_++o(1))
    \\
    =\left(\prod_{l=N_0}^n\overline{\eta^-_l}\right)
    \begin{pmatrix}
    h^+_{n+1}-z_n& h^-_{n+1}-z_n
    \end{pmatrix}
    \begin{pmatrix}
    1+o(1)
    \\
    o(1)
    \end{pmatrix}
\end{multline}
and analogously
\begin{equation}\label{}
    u^-_{n+1}-z_nu^-_n=\left(\prod_{l=N_0}^n\eta^-_l\right)\begin{pmatrix}
    h^+_{n+1}-z_n& h^-_{n+1}-z_n
    \end{pmatrix}
    \begin{pmatrix}
    o(1)
    \\
    1+o(1)
    \end{pmatrix}.
\end{equation}
From \eqref{z-n-k} and \eqref{h-n-k} we have
\begin{equation}
    u^+_{n_k+1}-z_{n_k}u^+_{n_k}=
    \left(\prod_{l=N_0}^{n_k}\overline{\eta^-_l}\right)(h^+_{n_k+1}-z_{n_k})(1+o(1)),
\end{equation}
\begin{equation}
    u^-_{n_k+1}-z_{n_k}u^-_{n_k}=
    \left(\prod_{l=N_0}^{n_k}\eta^-_l\right)o\left(\frac1{n_k^{\frac{\alpha}2}}\right),
\end{equation}
as $k\to\infty$, and with \eqref{H} this all implies that
\begin{equation}\label{}
    |u^+_{n_k+1}(\lambda)-z_{n_k}(\lambda)u^+_{n_k}(\lambda)|=\frac{2\sqrt{-\lambda}H(\lambda)+o(1)}{n_k^{\frac{\alpha}4}},
\end{equation}
\begin{equation}
    |u^-_{n_k+1}(\lambda)-z_{n_k}(\lambda)u^-_{n_k}(\lambda)|=o\left(\frac1{n_k^{\frac{\alpha}4}}\right).
\end{equation}
Therefore
\begin{equation}\label{}
    |P_{n_k+1}(\lambda)-z_{n_k}(\lambda)P_{n_k}(\lambda)|=\frac{2\sqrt{-\lambda}|\Psi(\lambda)|H(\lambda)+o(1)}{n_k^{\frac{\alpha}4}},\quad k\to\infty,
\end{equation}
uniformly in $\lambda\in[-R,-r]$.  Here we write only the asymptotics of absolute values, because, firstly, this is exactly what we need to proceed, and secondly, because its form should depend on $\alpha$ and cannot be written explicitly for all $\alpha\in(0,1)$, see Remark \ref{rem on the form of asymptotics}. Since $a_n=n^{\alpha}(1+o(1))$ due to \eqref{p-q from th}, this together with \eqref{sqrt 1-d^2} is enough to pass to the limit for the subsequence $\rho_{n_k}'$ using \eqref{rho n}:
\begin{equation}
    \rho'_{n_k}(\lambda)=\frac{\sqrt{1-d_{n_k}^2(\lambda)}}{\pi a_n|P_{n_k+1}(\lambda)-z_{n_k}(\lambda)P_{n_k}(\lambda)|^2}\to\frac1{4\pi\sqrt{-\lambda}|\Psi(\lambda)|^2H^2(\lambda)},
\end{equation}
and this limit is uniform in $\lambda\in[-R,-r]$. By Proposition \ref{prop weak limit} the spectral measure $\rho$ is absolutely continuous on $(-R,-r)$ and
\begin{equation}\label{}
    \rho'(\lambda)=\frac1{4\pi\sqrt{-\lambda}|\Psi(\lambda)|^2H^2(\lambda)}=\frac{\sqrt{-\lambda}H^2(\lambda)}{\pi|u^-_0(\lambda)|^2}\quad\text{ for a.a. }\lambda\in(-R,-r).
\end{equation}
Finally, note that $r$ can be chosen arbitrarily small and $R$ arbitrarily large to cover the whole $\mathbb R_-$.
\end{proof}

\section{Appendix A. Proof of Proposition \ref{prop density of J-N}}

Due to different numbering of matrix entries the form of the stabilized matrix $\mathcal J_N$ slightly differs from the form used in \cite{Aptekarev-Geronimo-2016}, and for this reason the formula from that work for the spectral density of the stabilized matrix is formally not directly applicable to our situation. This means that we need to provide a proof of Proposition \ref{prop density of J-N}.

\begin{proof}
 We will restrict ourselves to considering $\lambda\in\mathbb C_+\cup(b_N-2a_N,b_N+2a_N)$, because this is enough for the proof and makes it easy to avoid ambiguous or overcomplicated notations. Consider the eigenvector equation for the matrix $\mathcal J_N$. For $n\leqslant N$ it coincides with the equation for the matrix $\mathcal J$,
\begin{equation}\label{small n}
    a_{n-1}u_{n-1}+b_nu_n+a_nu_{n+1}=\lambda u_n,\quad n\leqslant N,
\end{equation}
for $n>N$ it has constant coefficients,
\begin{equation}\label{large n}
    a_Nu_{n-1}+b_Nu_n+a_Nu_{n+1}=\lambda u_n,\quad n>N.
\end{equation}
For every $\lambda\in\mathbb C_+\cup(b_N-2a_N,b_N+2a_N)$ it admits two pairs of solutions: the pair of sequences of polynomials $P_{N,n}(\lambda)$ and $Q_{N,n}(\lambda)$ of the first and of the second kind, respectively, and another pair of solutions $u^{\pm}_{N,n}(\lambda)$ which equal $z_N^{\mp n}(\lambda)$ for $n\geqslant N$ (note that $z_N(\lambda)$ and $z_N^{-1}(\lambda)$ are the characteristic roots of the equation \eqref{large n}) and are defined for $n<N$ by solving \eqref{small n} backwards. At the same time, the polynomials $P_{N,n}(\lambda)$ and $Q_{N,n}(\lambda)$ for the matrix $\mathcal J_N$ coincide with the polynomials $P_n(\lambda)$ and $Q_n(\lambda)$ for the matrix $\mathcal J$ for $n\leqslant N+1$, which can be seen from \eqref{small n} immediately. These two pairs are related by the following identities:
\begin{equation}\label{two pairs}
    P_{N,n}(\lambda)=\Phi_N(\lambda)u^+_{N,n}(\lambda)+\Phi^{(1)}_N(\lambda)u^-_{N,n}(\lambda),
    \quad
    Q_{N,n}(\lambda)=\Theta_N(\lambda)u^+_{N,n}(\lambda)+\Theta^{(1)}_N(\lambda)u^-_{N,n}(\lambda),
\end{equation}
 for all $\lambda\in\mathbb C_+\cup(b_N-2a_N,b_N+2a_N)$ and $n\geqslant1$ with some coefficients $\Phi_N(\lambda)$, $\Phi^{(1)}_N(\lambda)$, $\Theta_N(\lambda)$ and $\Theta^{(1)}_N(\lambda)$. By calculation of the Wronskian for large $n$ we have
\begin{equation}\label{}
    W\{u^+_{N,n}(\lambda),u^-_{N,n}(\lambda)\}=a_N\left(z_N(\lambda)-\frac1{z_N(\lambda)}\right),
\end{equation}
\begin{equation}\label{Phi}
    \Phi_N(\lambda)=\frac{W\{P_{N,n}(\lambda),u^-_{N,n}(\lambda)\}}{W\{u^+_{N,n}(\lambda),u^-_{N,n}(\lambda)\}}=\frac{P_N(\lambda)z_N^{N+1}(\lambda)-P_{N+1}(\lambda)z_N^N(\lambda)}{z_N(\lambda)-\frac1{z_N(\lambda)}},
\end{equation}
\begin{equation}
    \Theta_N(\lambda)=\frac{W\{Q_{N,n}(\lambda),u^-_{N,n}(\lambda)\}}{W\{u^+_{N,n}(\lambda),u^-_{N,n}(\lambda)\}}=\frac{Q_N(\lambda)z_N^{N+1}(\lambda)-Q_{N+1}(\lambda)z_N^N(\lambda)}{z_N(\lambda)-\frac1{z_N(\lambda)}},
\end{equation}
 which shows that the functions $\Phi_N$ and $\Theta_N$ are analytic in $\mathbb C_+$ and continuous in $\mathbb C_+\cup(b_N-2a_N,b_N+2a_N)$.
For $\lambda\in(b_N-2a_N,b_N+2a_N)$ and $n\geqslant N$ \eqref{two pairs} becomes
\begin{equation}\label{P,Q real interval}
    P_{N,n}(\lambda)=\frac{\Phi_N(\lambda)}{z_N^n(\lambda)}+\overline{\Phi_N(\lambda)}z_N^n(\lambda),
    \quad
    Q_{N,n}(\lambda)=\frac{\Theta_N(\lambda)}{z_N^n(\lambda)}+\overline{\Theta_N(\lambda)}z_N^n(\lambda).
\end{equation}
For $\lambda\in\mathbb C_+$ the solution
\begin{equation}\label{}
    Q_{N,n}(\lambda)+m_N(\lambda)P_{N,n}(\lambda)=\frac{\Theta_N(\lambda)+m_N(\lambda)\Phi_N(\lambda)}{z_N^n(\lambda)}+(\Theta^{(1)}_N(\lambda)+m_N(\lambda)\Phi^{(1)}_N(\lambda))z_N^n(\lambda)
\end{equation}
has to belong to $l^2$, hence $\Theta_N(\lambda)+m_N(\lambda)\Phi_N(\lambda)=0$ (recall that $|z_N(\lambda)|<1$ for $\lambda\in\mathbb C_+$), or
\begin{equation}\label{}
    m_N(\lambda)=-\frac{\Theta_N(\lambda)}{\Phi_N(\lambda)},\quad\lambda\in\mathbb C_+.
\end{equation}
This equality can be continued to $\lambda\in(b_N-2a_N,b_N+2a_N)$ where it implies \cite{Akhiezer-Glazman-1963} that
\begin{equation}\label{rho beta}
    \rho'_N(\lambda)=\frac{{\rm Im\,}m_N(\lambda)}{\pi}=\frac{\Phi_N(\lambda)\overline{\Theta_N(\lambda)}-\overline{\Phi_N(\lambda)}\Theta_N(\lambda)}{2\pi i|\Phi_N(\lambda)|^2}.
\end{equation}
Calculations of the Wronskian for $n=1$ and $n\geqslant N$ using \eqref{P,Q real interval} yield for $\lambda\in(b_N-2a_N,b_N+2a_N)$:
\begin{multline}\label{}
    1= W\{P_{N,n},Q_{N,n}\}=a_N\left(\frac{\Phi_N}{z_N^n}+\overline{\Phi_N}z_N^n\right)
    \left(\frac{\Theta_N}{z_N^{n+1}}+\overline{\Theta_N}z_N^{n+1}\right)
    \\
    -a_N\left(\frac{\Phi_N}{z_N^{n+1}}+\overline{\Phi_N}z_N^{n+1}\right)
    \left(\frac{\Theta_N}{z_N^{n}}+\overline{\Theta_N}z_N^{n}\right)
    =a_N(\Phi_N\overline{\Theta_N}-\overline{\Phi_N}\Theta_N)\left(z_N-\frac1{z_N}\right)
\end{multline}
and therefore
\begin{equation}\label{}
    \Phi_N(\lambda)\overline{\Theta_N(\lambda)}-\overline{\Phi_N(\lambda)}\Theta_N(\lambda)=\frac1{a_N\left(z_N(\lambda)-\frac1{z_N(\lambda)}\right)}.
\end{equation}
So we have
\begin{equation}\label{rho beta'}
    \rho_N'(\lambda)=\frac1{2\pi i a_N\left(z_N(\lambda)-\frac1{z_N(\lambda)}\right)|\Phi_N(\lambda)|^2},\quad\lambda\in(b_N-2a_N,b_N+2a_N).
\end{equation}
Taking the absolute value of \eqref{Phi}, we arrive at the equality
\begin{equation}\label{}
    |\Phi_N(\lambda)|
    =\frac{|P_{N+1}(\lambda)-z_N(\lambda)P_N(\lambda)|}{\left|z_N(\lambda)-\frac1{z_N(\lambda)}\right|},
    \quad\lambda\in(b_N-2a_N,b_N+2a_N).
\end{equation}
Together with the formula \eqref{z-N} this gives
\begin{equation}\label{}
    \rho_N'(\lambda)=\frac{\left|z_N(\lambda)-\frac1{z_N(\lambda)}\right|}{2\pi a_N|P_{N+1}(\lambda)-z_N(\lambda)P_N(\lambda)|^2}=\frac{\sqrt{1-\left(\frac{\lambda-b_N}{2a_N}\right)^2}}{\pi a_N|P_{N+1}(\lambda)-z_N(\lambda)P_N(\lambda)|^2},
\end{equation}
which completes the proof.
\end{proof}

\section{Appendix B. Spectral density in the non-critical case (Apte\-ka\-rev--Ge\-ro\-ni\-mo theorem revisited)}

In this appendix we consider the class of Jacobi matrices from the paper by Aptekarev and Geronimo \cite{Aptekarev-Geronimo-2016} and show how to prove their formula for the spectral density using the technique of the proof in the critical case above. This gives a somewhat new proof of the known fact. Note that the application of our technique in the non-critical case is much simpler than in the critical one.

We use the same notation for the entries of a Jacobi matrix $\{a_n\}_{n=1}^{\infty}$ and $\{b_n\}_{n=1}^{\infty}$ and impose different assumptions on them. We hope that this will not lead to misunderstanding.

    \begin{theorem}
    Let the entries $a_n>0$ and $b_n\in\mathbb R$, $n\in\mathbb N$, of the Jacobi matrix $\mathcal J$ be such that
    \begin{equation}\label{app b bounded variation}
        \left\{\frac{b_n}{a_n}\right\}_{n=1}^{\infty},
        \left\{\frac1{a_n}\right\}_{n=1}^{\infty},
        \left\{\frac{a_{n-1}}{a_n}\right\}_{n=1}^{\infty}
    \end{equation}
    are sequences of bounded variation,
    \begin{equation}\label{app b limits of entries}
        \frac{b_n}{a_n}\to2d,\quad\frac1{a_n}\to0,\quad\frac{a_{n-1}}{a_n}\to1\text{ as }
        n\to\infty
    \end{equation}
    with
    \begin{equation}\label{}
        d\in(-1,1)
    \end{equation}
    and
    \begin{equation}\label{}
        \sum_{n=1}^{\infty}\frac1{a_n}=\infty.
    \end{equation}
    Then for every $\lambda\in\overline{\mathbb C}_+$ the eigenvector equation for $\mathcal J$,
    \begin{equation}\label{app b eigenfunction equation}
        a_{n-1}u_{n-1}+b_nu_n+a_nu_{n+1}=\lambda u_n,\quad n\geqslant2,
    \end{equation}
    has a solution $u_n^-(\lambda)$ with the asymptotics
    \begin{equation}\label{app b u-}
        u_n^-(\lambda)=\left(\prod_{l=2}^n\mu_l^-(\lambda)\right)(1+o(1)),\quad n\to\infty,
    \end{equation}
    where
    \begin{equation}\label{app b mu}
        \mu^-_n(\lambda):=\frac{\lambda-b_n}{2a_n}- i\sqrt{\frac{a_{n-1}}{a_n}-\left(\frac{\lambda-b_n}{2a_n}\right)^2}\to-d-i\sqrt{1-d^2},\quad n\to\infty.
    \end{equation}
    The asymptotics \eqref{app b u-} is uniform in every compact set $K\subset\overline{\mathbb C}_+$. For any $n$ the vector component $u^-_n$ is analytic in $\mathbb C_+$ and continuous in $\overline{\mathbb C}_+$. The limit
    \begin{equation}\label{app b M}
        M(\lambda ):=\lim_{n\to\infty}\sqrt{a_n}\prod_{l=2}^n|\mu_l^-(\lambda)|,\quad\lambda\in\mathbb R,
    \end{equation}
    exists and is finite, nonzero and continuous in $\lambda\in\mathbb R$. The sequence $u^+_n(\lambda):=\overline{u^-_n(\lambda)}$ for $\lambda\in\mathbb R$ is also a solution to the equation \eqref{app b eigenfunction equation} and the Wronskian of solutions $u^{\pm}(\lambda)=\{u^{\pm}_n(\lambda)\}_{n=1}^{\infty}$
    \begin{equation}\label{app b Wronskian u+u-}
        W\{u^+(\lambda),u^-(\lambda)\}=-2i\sqrt{1-d^2}M^2(\lambda),
    \end{equation}
    therefore $u^+(\lambda)$ and $u^-(\lambda)$ are linearly independent. The orthogonal polynomials can be expressed for $\lambda\in\mathbb R$ as
    \begin{equation}\label{app b decomposition}
        P_n(\lambda)=\Psi(\lambda)u_n^+(\lambda)+\overline{\Psi(\lambda)}u_n^-(\lambda),
    \end{equation}
    where
    \begin{equation}\label{app b Psi}
        \Psi(\lambda)=\frac{u_0^-(\lambda)}{2i\sqrt{1-d^2}M^2(\lambda)},
    \end{equation}
    $u_0^-(\lambda):=(\lambda-b_1)u_1^-(\lambda)-a_1u_2^-(\lambda)$ (assuming formally in \eqref{app b eigenfunction equation} that $a_0:=1$). Finally,
    \begin{equation}\label{app b density}
        \rho'(\lambda)
        =
        \frac1{4\pi\sqrt{1-d^2}|\Psi(\lambda)|^2M^2(\lambda)}
        =
        \frac{\sqrt{1-d^2}M^2(\lambda)}{\pi|u^-_0(\lambda)|^2},\quad\lambda\in\mathbb R.
    \end{equation}
    \end{theorem}

    \begin{remark}
     The critical case corresponds to $d=\pm1$, and the formula \eqref{app b density} then fails. For $|d|>1$ the spectrum of $\mathcal J$ is discrete, since we are in the situation of dominating main diagonal. This can be shown by estimating quadratic forms of truncated matrices \cite{Janas-Naboko-2001,Szwarz-2002}.
    \end{remark}

    \begin{remark}
    Note that the conditions \eqref{app b bounded variation}--\eqref{app b limits of entries} are much weaker compared to the condition \eqref{p-q from th} in the critical case. Surely, it is not surprising and, moreover, the similar situation happens for discrete Schr\"odinger operator with decreasing potential at the edges of the essential spectrum, the points $\lambda=\pm2$.
    \end{remark}

\begin{proof}
From the assumption \eqref{app b limits of entries} it immediately follows that in the formula for the spectral density \eqref{density of J-N} of the stabilized matrix $\mathcal J_n$,
\begin{equation}\label{app b sieved density}
    \rho'_n(\lambda)=\frac{\sqrt{1-\left(\frac{\lambda-b_n}{2a_n}\right)^2}}{\pi a_n|P_{n+1}(\lambda)-z_n(\lambda)P_n(\lambda)|^2},
\end{equation}
the numerator converges to $\sqrt{1-d^2}$ as $n\to\infty$.  Note that the complex numbers
\begin{equation}\label{}
    \mu^{\pm}_n(\lambda):=\frac{\lambda-b_n}{2a_n}\pm i\sqrt{\frac{a_{n-1}}{a_n}-\left(\frac{\lambda-b_n}{2a_n}\right)^2},
\end{equation}
are exactly the eigenvalues of the transfer-matrix
\begin{equation}\label{}
    B_n(\lambda)=
    \begin{pmatrix}
    0 & 1 \\
    -\frac{a_{n-1}}{a_n} & \frac{\lambda-b_n}{a_n}
    \end{pmatrix}.
\end{equation}
Here the branches for $\lambda\in\overline{\mathbb C}_+$ are chosen so that $\sqrt{\frac{a_n}{a_{n-1}}}\mu^+_n(\lambda)\in\overline{\mathbb C}_+\backslash\mathbb D$ and $\sqrt{\frac{a_n}{a_{n-1}}}\mu^-_n(\lambda)\in(\overline{\mathbb C}_-\cap\overline{\mathbb D})\backslash\{0\}$, $\mathbb D$ denotes the open unit disc.
The determinant of the transfer-matrix equals
\begin{equation}\label{}
  \mu^+_n(\lambda)\mu^-_n(\lambda)=\frac{a_{n-1}}{a_n},\quad n\geqslant2.
\end{equation}
From \eqref{app b limits of entries} we also see that
\begin{equation}\label{app b limit of mu-pm}
    B_n(\lambda)\to
    \begin{pmatrix}
      0 & 1 \\
      -1 & -2d \\
    \end{pmatrix},
    \quad
    \mu^{\pm}_n(\lambda)\to-d\pm i\sqrt{1-d^2},
\end{equation}
\begin{equation}\label{}
    z_n(\lambda)\to-d-i\sqrt{1-d^2},\quad \frac1{z_n(\lambda)}\to-d+i\sqrt{1-d^2},
\end{equation}
as $n\to\infty$, and for $\lambda\in\mathbb R$ the eigenvalues of $B_n(\lambda)$ are in the elliptic case  for large $n$:
\begin{equation}\label{}
    \mu^+_n(\lambda)=\overline{\mu^-_n(\lambda)},\quad|\mu^-_n(\lambda)|=|\mu^+_n(\lambda)|=\sqrt{\frac{a_{n-1}}{a_n}}.
\end{equation}
Therefore the limit of the sequence $\lim_{n\to\infty}\sqrt{a_n}\prod_{l=2}^{n}|\mu^-_l(\lambda)|$ as $n\to\infty$ exists for every $\lambda\in\mathbb R$ and is a continuous function of $\lambda$ without zeros. Denote
\begin{equation}\label{app b z}
    z:=-d-i\sqrt{1-d^2}.
\end{equation}
At the real points
\begin{equation}\label{}
    \lambda^{\pm}_n:=b_n\pm2\sqrt{a_{n-1}a_n}
\end{equation}
we have $\mu^+_n\equiv\mu^-_n$ and the transfer-matrix $B_n(\lambda)$ is not diagonalizable, while it is diagonalizable for all other $\lambda\in\mathbb C$. Denote
\begin{equation}\label{}
    X_N:=\{\lambda^{\pm}_n,n\geqslant N\}.
\end{equation}
Fix an arbitrary compact set $K\subset\overline{\mathbb C}_+$. Since $\lambda^{\pm}_n\sim2a_n(d\pm1)$ as $n\to\infty$, for every $d\in(-1,1)$ there exists $N(K)\in\mathbb N$ such that $K\cap X_{N(K)}=\emptyset$. This means that for all $n\geqslant N(K)$ and $\lambda\in K$ one can diagonalize the matrix $B_n(\lambda)$,
\begin{equation}\label{}
    B_n(\lambda)=
    \begin{pmatrix}
      1 & 1 \\
      \mu^+_n(\lambda) & \mu^-_n(\lambda) \\
    \end{pmatrix}
    \begin{pmatrix}
      \mu^+_n(\lambda) & 0 \\
      0 & \mu^-_n(\lambda) \\
    \end{pmatrix}
    \begin{pmatrix}
      1 & 1 \\
      \mu^+_n(\lambda) & \mu^-_n(\lambda) \\
    \end{pmatrix}^{-1},\quad n\geqslant N(K).
\end{equation}
The substitution
\begin{equation}\label{app b u to v}
    \vec u_n
    =
    \begin{pmatrix}
      1 & 1 \\
      \mu^+_n(\lambda) & \mu^-_n(\lambda) \\
    \end{pmatrix}
    \vec v_n
\end{equation}
transforms for $n\geqslant N(K)$ the system
\begin{equation}\label{app b system u}
    \vec u_{n+1}=B_n(\lambda)\vec u_n\,,
\end{equation}
which is equivalent to the eigenfunction equation \eqref{app b eigenfunction equation}, to the system
\begin{multline}\label{app b system v beta}
    \vec v_{n+1}=
    \begin{pmatrix}
    1 & 1 \\
    \mu_{n+1}^+(\lambda) & \mu_{n+1}^-(\lambda)
    \end{pmatrix}^{-1}
    B_n(\lambda)
    \begin{pmatrix}
    1 & 1 \\
    \mu^+_n(\lambda) & \mu^-_n(\lambda)
    \end{pmatrix}
    \vec v_n
    \\
    =
    \begin{pmatrix}
      1 & 1 \\
      \mu_{n+1}^+(\lambda) & \mu_{n+1}^-(\lambda)
    \end{pmatrix}^{-1}
    \begin{pmatrix}
      1 & 1 \\
      \mu^+_n(\lambda) & \mu^-_n(\lambda)
    \end{pmatrix}
    \begin{pmatrix}
      \mu^+_n(\lambda) & 0 \\
      0 & \mu^-_n(\lambda)
    \end{pmatrix}
    \vec v_n,\quad n\geqslant N(K).
\end{multline}
Using the condition \eqref{app b bounded variation} one can show by an elementary calculation that
\begin{equation}\label{}
    \left\{
    \sup_{\lambda\in K}
    \left\|
    \begin{pmatrix}
      1 & 1 \\
      \mu_{n+1}^+(\lambda) & \mu_{n+1}^-(\lambda)
    \end{pmatrix}^{-1}
    \begin{pmatrix}
      1 & 1
      \\
      \mu^+_n(\lambda) & \mu^-_n(\lambda)
    \end{pmatrix}
    -I
    \right\|
    \right\}_{n=N(K)}^{\infty}\in l^1
\end{equation}
(here we essentially use that $|d|<1$). Therefore the system \eqref{app b system v beta} can be written as
\begin{equation}\label{app b system v}
    \vec v_{n+1}=
    \left(
    \begin{pmatrix}
      \mu^+_n(\lambda) & 0 \\
      0 & \mu^-_n(\lambda) \\
    \end{pmatrix}
    +R_n(\lambda)
    \right)
    \vec v_n,\quad n\geqslant N(K),
\end{equation}
 with $\{\sup_{\lambda\in K}\|R_n(\lambda)\|\}_{n=N(K)}^{\infty}\in l^1$. Consider the second substitution $\vec v_n\to\vec x_n$:
\begin{equation}\label{app b trasformation v to x}
    \vec v_n=\left(\prod_{l=N(K)}^{n-1}\sqrt{\mu^+_l(\lambda)\mu^-_l(\lambda)}\right)\vec x_n,\quad n\geqslant N(K),
\end{equation}
this transforms the system \eqref{app b system v} to the system
\begin{equation}\label{app b system x}
    \vec x_{n+1}=
    \left(
    \begin{pmatrix}
    \sqrt{\frac{\mu^+_n(\lambda)}{\mu^-_n(\lambda)}} & 0 \\
    0 & \sqrt{\frac{\mu^-_n(\lambda)}{\mu^+_n(\lambda)}}
    \end{pmatrix}
    +\frac{R_n(\lambda)}{\sqrt{\mu^+(\lambda)\mu^-(\lambda)}}
    \right)
    \vec x_n,\quad n\geqslant N(K).
\end{equation}
Lemma \ref{lemma asymptotics} is applicable to this system for $n\geqslant N(K)$: there exists $C(K)>0$ such that for every $p,q\in\mathbb N$ such that $N(K)\leqslant p\leqslant q$
\begin{equation}\label{}
    \prod_{l=p}^q\left|\sqrt{\frac{\mu^+_l(\lambda)}{\mu^-_l(\lambda)}}\right|\geqslant\frac1{C(K)}
\end{equation}
(provided by the fact that we choose $|\mu^+_n(\lambda)|\geqslant|\mu^-_n(\lambda)|$ for large $n$) and
\begin{equation}\label{}
    \left\{\sup_{\lambda\in K}\frac{\left\|R_n(\lambda)\right\|}{|\mu^-_n(\lambda)|}\right\}_{n=N(K)}^{\infty}\in l^1,
\end{equation}
which gives the condition \eqref{condition Levinson 2} of Lemma \ref{lemma asymptotics}. By the lemma there exists a solution $\vec x^{\,K,-}(\lambda)$ of the system \eqref{app b system x}
\begin{equation}\label{}
    \vec x^{\,K,-}_n(\lambda)=\left(\prod_{l=N(K)}^{n-1}\sqrt{\frac{\mu^-_l(\lambda)}{\mu^+_l(\lambda)}}\right)(\vec e_-+o(1)),\quad n\to\infty,
\end{equation}
which is analytic in ${\rm int\,}K$, continuous in $K$ and has uniform in $\lambda\in K$ asymptotics. Correspondingly the system \eqref{app b system u} has a solution
\begin{equation}\label{}
    \vec u^{\,K,-}_n(\lambda):=
    \begin{pmatrix}
    1 & 1 \\
    \mu^+_n(\lambda) & \mu^-_n(\lambda)
    \end{pmatrix}
    \left(\prod_{l=N(K)}^{n-1}\sqrt{\mu^+_l(\lambda)\mu^-_l(\lambda)}\right)\vec x^{\,K,-}_n(\lambda)
\end{equation}
with the same properties,
\begin{multline}\label{}
    \vec u^{\,K,-}_n(\lambda)
    =
    \begin{pmatrix}
    u^{K,-}_{n-1}(\lambda) \\
    u^{K,-}_n(\lambda)
    \end{pmatrix}
    =
    \begin{pmatrix}
    1 & 1 \\
    \mu^+_n(\lambda) & \mu^-_n(\lambda)
    \end{pmatrix}
    \left(\prod_{l=N(K)}^{n-1}\mu^-_l(\lambda)\right)
    \left(\vec e_-+o(1)\right)
    \\
    =\left(\prod_{l=N(K)}^{n-1}\mu^-_l(\lambda)\right)
    \left(
    \begin{pmatrix}
    1 \\
    z
    \end{pmatrix}
    +o(1)
    \right),\quad n\to\infty.
\end{multline}
This solution is formally defined only for $n\geqslant N(K)$, but obviously exists for all $n\geqslant2$ retaining its properties, because the matrices $B_n(\lambda)$ and $B^{-1}_n(\lambda)$ are entire functions of $\lambda$ for all $n$. Let us observe that the solution proportional to $\vec u^{\,K,-}_n(\lambda)$
\begin{equation}\label{}
    \vec u^{\,-}_n(\lambda):=
    \left(\prod_{l=2}^{N(K)-1}\mu^-_l(\lambda)\right)
    \vec u^{\,K,-}_n(\lambda),
\end{equation}
\begin{equation}\label{app b u- definition}
    \vec u^{\,-}_n(\lambda)=
    \begin{pmatrix}
    u^-_{n-1}(\lambda) \\
    u^-_n(\lambda)
    \end{pmatrix}
    =
    \left(\prod_{l=2}^{n-1}\mu^-_l(\lambda)\right)
    \left(
    \begin{pmatrix}
    1 \\
    z
    \end{pmatrix}
    +o(1)
    \right),\quad n\to\infty,
\end{equation}
does not depend on $K$ and hence is analytic in $\mathbb C_+$ and continuous in $\overline{\mathbb C}_+$. It is enough to show that for any compact sets $K_1,K_2$ such that $K_1\subset K_2\subset\overline{\mathbb C}_+$ for every $\lambda\in K_1$ solutions $\vec u^{\,K_1,-}(\lambda)$ and $\vec u^{\,K_2,-}(\lambda)$ are proportional (and not only have proportional asymptotics). To this end consider the Wronskian: for any $n\geqslant1$ and $\lambda\in K_1$
\begin{multline}
    W\{u^{K_1,-}(\lambda),u^{K_2,-}(\lambda)\}
    =
    a_n
    \det
    \begin{pmatrix}
    u^{K_1,-}_n(\lambda) & u^{K_2,-}_n(\lambda) \\
    u^{K_1,-}_{n+1}(\lambda) & u^{K_2,-}_{n+1}(\lambda) \\
    \end{pmatrix}
    =
    \lim_{n\to\infty}
    a_n
    \det
    \begin{pmatrix}
    u^{K_1,-}_n(\lambda) & u^{K_2,-}_n(\lambda) \\
    u^{K_1,-}_{n+1}(\lambda) & u^{K_2,-}_{n+1}(\lambda) \\
    \end{pmatrix}
    \\
    =
    \lim_{n\to\infty}
    \left(
    a_n
    \left(\prod_{l=N(K_1)}^n\mu^-_l(\lambda)\right)\left(\prod_{l=N(K_2)}^n\mu^-_l(\lambda)\right)
    \det
    \left(
    \begin{pmatrix}
    1 & 1 \\
    z & z
    \end{pmatrix}
    +o(1)
    \right)
    \right)
    =0
\end{multline}
using \eqref{app b M}. For $\lambda\in\mathbb R$ the sequence $u^+_n(\lambda):=\overline{u^-_n(\lambda)}$ is a solution of \eqref{app b system u} and, since $|z|=1$, $\overline z=\frac1z$,
\begin{multline}
    W\{u^+(\lambda),u^-(\lambda)\}
    =a_n
    \det
    \begin{pmatrix}
    u^+_n(\lambda) & u^-_n(\lambda) \\
    u^+_{n+1}(\lambda) & u^-_{n+1}(\lambda) \\
    \end{pmatrix}
    \\
    =
    \lim_{n\to\infty}a_n
    \det
    \left[
    \left(
    \begin{pmatrix}
    1 & 1 \\
    \frac1z & z
    \end{pmatrix}
    +o(1)
    \right)
    \prod_{l=2}^n
    \begin{pmatrix}
    \overline{\mu^-_l(\lambda)} & 0 \\
    0 & \mu^-_l(\lambda)
    \end{pmatrix}
    \right]
    \\
    =\lim_{n\to\infty}a_n\left(z-\frac1z+o(1)\right)\left(\prod_{l=2}^n|\mu^-_l(\lambda)|^2\right)
    =-2i\sqrt{1-d^2}M^2(\lambda).
\end{multline}
Therefore $u^+(\lambda)$ and $u^-(\lambda)$ form a basis of solutions of \eqref{app b eigenfunction equation}. We have for $\lambda\in\mathbb R$
\begin{equation}\label{app b decomposition beta}
    P_n(\lambda)=\Psi(\lambda)u_n^+(\lambda)+\overline{\Psi(\lambda)}u_n^-(\lambda)
\end{equation}
with
\begin{equation}\label{}
    \Psi(\lambda)=\frac{W\{P(\lambda),u^-(\lambda)\}}{W\{u^+(\lambda),u^-(\lambda)\}}=\frac{u^-_0(\lambda)}{2i\sqrt{1-d^2}M^2(\lambda)}.
\end{equation}
Recall that $u^-_0(\lambda)=(\lambda-b_1)u_1^-(\lambda)-a_1u_2^-(\lambda)$, then
\begin{multline}\label{}
    P_{n+1}(\lambda)-z_n(\lambda)P_n(\lambda)=\Psi(\lambda)(u^+_{n+1}(\lambda)-z_n(\lambda)u^+_n(\lambda))+\overline{\Psi(\lambda)}(u^-_{n+1}(\lambda)-z_n(\lambda)u^+_n(\lambda))
    \\
    =\Psi(\lambda)\left(\prod_{l=2}^n\overline{\mu^-_l(\lambda)}\right)(\overline{\mu^-_{n+1}(\lambda)}-z_n(\lambda)+o(1))
    +\overline{\Psi(\lambda)}\left(\prod_{l=2}^n\mu^-_l(\lambda)\right)(\mu^-_{n+1}(\lambda)-z_n(\lambda)+o(1))
    \\
    =\Psi(\lambda)\left(\prod_{l=2}^n\overline{\mu^-_l(\lambda)}\right)\left(\frac1z-z+o(1)\right),
\end{multline}
since $\mu^-_{n+1}(\lambda)-z_n(\lambda)\to0$ and $\overline{\mu^-_{n+1}(\lambda)}-z_n(\lambda)\to\frac1z-z$ as $n\to\infty$. Therefore
\begin{equation}\label{}
    \sqrt{a_n}|P_{n+1}(\lambda)-z_n(\lambda)P_n(\lambda)|\to|\Psi(\lambda)|\lim_{n\to\infty}\left(\sqrt{a_n}\prod_{l=2}^n|\mu^-_l(\lambda)|\right)\left|\frac1z-z\right|=2|\Psi(\lambda)|M(\lambda)\sqrt{1-d^2},
\end{equation}
and the denominator of \eqref{app b sieved density} converges to $4\pi|\Psi(\lambda)|^2M^2(\lambda)(1-d^2)$ as $n\to\infty$. By Proposition \ref{prop weak limit} we arrive at the formula \eqref{app b density}. This completes the proof.
\end{proof}

\section{Acknowledgements}

This work was supported by RFBR 19-01-00657A, RFBR 19-01-00565A and RFBR 17-01-00529A grants and by the Knut and Alice Wallenberg Foundation. The first author appreciates hospitality of the Mittag-Lefler Institute, where part of this work was done.


\begin{thebibliography}{9}





\bibitem{Akhiezer-1965}
N. I. Akhiezer, The classical moment problem and some related questions in analysis, {\it Oliver \& Boyd, Edinburgh}, 1965.

\bibitem{Akhiezer-Glazman-1963}
N. I. Akhiezer, I. M. Glazman, Theory of operators in Hilbert space. {\it Frederick Ungar, New York}, 1963.

\bibitem{Van Assche-Geronimo-1988}
W. Van Assche, J. S. Geronimo, Asymptotics for orthogonal polynomials on and off the essential spectrum, {\it J. Approx. Theory} 55, 1988, 220--231.

\bibitem{Aptekarev-Geronimo-2016}
A. I. Aptekarev, J. S. Geronimo, Measures for orthogonal polynomials with unbounded recurrence coefficients, {\it J. Approx. Theory} 207, 2016, 339--347.

\bibitem{Benzaid-Lutz-1987}
Z. Benzaid, D. A. Lutz, Asymptotic representation of solutions of perturbed systems of linear difference equations, {\it Stud. Appl. Math.} 77(3), 1987, 195--221.

\bibitem{Bodine-Lutz-2015}
S. Bodine, D. A. Lutz, Asymptotic integration of differential and difference equations, {\it Springer, Berlin}, 2015.

\bibitem{Coddington-Levinson-1955}
E. A. Coddington, N. Levinson, Theory of ordinary differential equations, {\it McGraw-Hill, New York}, 1955.

\bibitem{Damanik-Naboko-2007}
D. Damanik, S. Naboko, Unbounded Jacobi matrices at critical coupling, {\it J. Approx. Theory} 145(2), 2007, 221--236.

\bibitem{Gilbert-Pearson-1987}
D. J. Gilbert, D. B. Pearson, On subordinacy and analysis of the spectrum of one-dimensional Schr\"odinger operators, {\it J. Math. Anal. Appl.} 128(1), 1987, 30--56.

\bibitem{Harris-Lutz-1975}
W. A. Harris, D. A. Lutz, Asymptotic integration of adiabatic oscillators, {\it J. Math. Anal. Appl.} 51, 1975, 76--93.

\bibitem{Janas-2006}
J. Janas, The asymptotic analysis of generalized eigenvectors of some Jacobi operators. Jordan box case, {\it J. Difference Equ. Appl.} 12(6), 2006, 597--618.

\bibitem{Janas-Moszynski-2003}
J. Janas, M. Moszynski, Spectral properties of Jacobi matrices by asymptotic analysis, {\it J. Approx. Theory} 120(2), 2003, 309--336.

\bibitem{Janas-Naboko-2001}
J. Janas, S. Naboko, Multithreshold spectral phase transitions for a class of Jacobi matrices, {\it Recent Advances in Operator Theory. Operator Theory: Advances and Applications, vol. 124, Birk\"auser, Basel}, 2001, 267--285.

\bibitem{Janas-Naboko-2002}
J. Janas, S. Naboko, Spectral analysis of selfadjoint Jacobi matrices with periodically modulated entries, {\it J. Funct. Anal.} 191(2), 2002, 318--342.

\bibitem{Janas-Naboko-Sheronova-2009}
J. Janas, S. Naboko, E. Sheronova, Asymptotic behavior of generalized eigenvectors of Jacobi matrices in the critical (``double root'') case, {\it Z. Anal. Anwend.} 28(4), 2009, 411--430.

\bibitem{Janas-Simonov-2010}
J. Janas, S. Simonov, Weyl--Titchmarsh type formula for discrete Schr\"odinger operator with Wigner--von Neumann potential, {\it Studia Math.} 201(2), 2010, 167--189.

\bibitem{Khan-Pearson-1992}
S. Khan, D. B. Pearson, Subordinacy and spectral theory for infinite matrices, {\it Helvetica Physica Acta} 65(4), 1992, 505--527.

\bibitem{Kooman-2007}
R.-J. Kooman, An asymptotic formula for solutions of linear second-order difference equations with regularly behaving coefficients, {\it J. Differ. Equ. Appl.} 13(11), 2007, 1037--1049.

\bibitem{Kurasov-Simonov-2014}
P. Kurasov, S. Simonov, Weyl--Titchmarsh-type formula for periodic Schr\"odinger operator with Wigner--von Neumann potential, {\it Proc. Roy. Soc. Edinburgh Sect. A} 143A, 2014, 401--425.

\bibitem{Mate-Nevai-Totik-1985}
A. M\'at\'e, P. Nevai, V. Totik, Asmptotics for orthogonal polynomials defined by a recurrence relation, {\it Constr. Approx.} 1, 1985, 231--248.

\bibitem{Naboko-Simonov-2010}
S. Naboko, S. Simonov, Spectral analysis of a class of Hermitian Jacobi matrices in a critical (double root) hyperbolic case, {\it Proc. Edinb. Math. Soc. (2)} 53(1), 2010, 239--254.

\bibitem{Naboko-Simonov-2012}
S. Naboko, S. Simonov, Zeroes of the spectral density of the periodic Schr\"odinger operator with Wigner--von Neumann potential, {\it Math. Proc. Cambridge Philos. Soc.} 153(1), 2012, 33--58.

\bibitem{Rudin-1987}
W. Rudin, Real and complex analysis, {\it McGraw-Hill, New York}, 1987.

\bibitem{Silva-2007}
L. O. Silva, Uniform and smooth Benzaid--Lutz type theorems and applications to Jacobi matrices, {\it Operator theory, analysis and mathematical physics. Operator Theory: Advances and Applications, vol. 174, Birkh\"auser, Basel}, 2007, 173--186.

\bibitem{Simonov-2007}
S. Simonov, An example of spectral phase transition phenomenon in a class of Jacobi matrices with periodically modulated weights. {\it Operator theory, analysis and mathematical physics. Operator Theory: Advances and Applications, vol. 174, Birkh\"auser, Basel}, 2007, 187--203.

\bibitem{Simonov-2009}
S. Simonov, Weyl--Titchmarsh type formula for Hermite operator with small perturbation, {\it Opuscula Math.} 29(2), 2009, 187--207.

\bibitem{Simonov-2012}
S. Simonov, Zeroes of the spectral density of discrete Schr\"odinger operator with Wigner--von Neumann potential, {\it Integral Equations Operator Theory} 73(3), 2012, 351--364.

\bibitem{Simonov-2016}
S. Simonov, Zeroes of the spectral density of the Schr\"odinger operator with the slowly decaying Wigner--von Neumann potential, {\it Math. Z.} 284, 2016, 335--411.

\bibitem{Swiderski-2016}
G. \'Swiderski, Spectral properties of unbounded Jacobi matrices with almost monotonic weights, {\it Constr. Approx.} 44, 2016, 141--157.

\bibitem{Swiderski-Trojan-2017}
G. \'Swiderski, B. Trojan, Periodic perturbations of unbounded Jacobi matrices I: asymptotics of generalized eigenvectors,  {\it J. Approx. Theory} 216, 2017, 38--66.

\bibitem{Swiderski-2017}
G. \'Swiderski, Periodic perturbations of unbounded Jacobi matrices II: formulas for density, {\it J. Approx. Theory} 216, 2017, 67--85.

\bibitem{Swiderski-2018}
G. \'Swiderski, Periodic perturbations of unbounded Jacobi matrices III: the soft edge regime, {\it J. Approx. Theory} 233, 2018, 1--36.

\bibitem{Szwarz-2002}
R. Szwarc, Absolute continuity of spectral measure for certain unbounded Jacobi matrices, {\it Advanced Problems in Constructive Approximation. International Series of Numerical Mathematics, vol. 142, Birkh\"auser, Basel}, 2002, 255--262.

\bibitem{Titchmarsh-1962}
E. C. Titchmarsh, Eigenfunctions expansions, I, {\it Clarendon Press, Oxford}, 1962.





\end{thebibliography}
\end{document}